\documentclass{amsart}

\usepackage{amssymb}
\usepackage[cp1250]{inputenc}
\usepackage{picture}
\usepackage{color}
\usepackage{tikz}
\usepackage{tikz}
\usetikzlibrary{matrix}

\newtheorem{theorem}{Theorem}[section]
\newtheorem{lemma}[theorem]{Lemma}
\newtheorem{example}[theorem]{Example}
\newtheorem{proposition}[theorem]{Proposition}
\newtheorem{problem}[theorem]{Problem}
\newtheorem{corollary}[theorem]{Corollary}

\newtheorem{remark}[theorem]{Remark}

\newcommand{\N}{\mathbb N}

\newcommand{\R}{\mathbb R}

\newcommand{\on}{\operatorname}
\author{Szymon G\l \c ab}
\address{Institute of Mathematics, \L \'od\'z University of Technology,
W\'olcza\'nska 215, 93-005 \L \'od\'z, Poland}
\email {szymon.glab@p.lodz.pl}

\author{Jacek Marchwicki}
\address{
Department of Complex Analysis,
Faculty of Mathematics and Computer Science,
University of Warmia and Mazury in Olsztyn,
Słoneczna 54,
10-710 Olsztyn,
Poland}
\email {marchewajaclaw@gmail.com}

\title[Cardinal functions of purely atomic measures]{Cardinal functions of purely atomic measures}


\begin{document}

\begin{abstract}
Let $\mu$ be a purely atomic measure. By $f_\mu:[0,\infty)\to\{0,1,2,\dots,\omega,\mathfrak{c}\}$ we denote its cardinal function $f_{\mu}(t)=\vert\{A\subset\N:\mu(A)=t\}\vert$. We study the problem for which sets $R\subset\{0,1,2,\dots,\omega,\mathfrak{c}\}$ there is a measure $\mu$ such that $R$ is the range of $f_\mu$. We are also interested in the set-theoretic and topological properties of the set of $\mu$-values which are obtained uniquely. 
\end{abstract}

\maketitle

\section{Introduction}

Let $\mu$ be a purely atomic finite measure. Following \cite{BGM} we may assume that $\mu$ is defined on $\N:=\{1,2,3,\dots\}$  and $\mu(\{k\})\geq \mu({\{k+1\}})>0$. Throughout the paper we assume that measures are always of this type. By the range of $\mu$ we understand the set $\text{rng}(\mu):=\{\mu(E):E\subset\N\}$. To simplify the notation let $x_n=\mu(\{n\})$ be a measure of the $n$-th largest atom of $\mu$. Note that 
\[
\text{rng}(\mu)=\Big\{\sum_{n=1}^\infty\varepsilon_nx_n:(\varepsilon_n)\in\{0,1\}^\N\Big\}.
\]
The latter set is also denoted by $A(x_n)$ and it is called the achievement set of $(x_n)$ (see \cite{Jones}). We can identify a measure $\mu$ with the function $\{0,1\}^\N\ni(\varepsilon_n)\mapsto\sum_{n=1}^\infty\varepsilon_nx_n$. This is the continuous mapping from the Cantor space $\{0,1\}^\N$ to the real line, and therefore $\mu^{-1}(t)$ is a closed subset of the Cantor space $\{0,1\}^\N$. Thus the cardinality $\vert\mu^{-1}(t)\vert$ belongs to the set $\{0,1,2,\dots,\omega,\mathfrak{c}\}$ where $\omega$ stands for the first infinite cardinal while $\mathfrak{c}$ for the continuum \cite[Section 6]{Kechris}. By $f_\mu$ we denote the cardinal function $f_{\mu}(t)=\vert\mu^{-1}(t)\vert=\vert\{A\subset\N:\mu(A)=t\}\vert$; we will write $f$ instead of $f_\mu$ if it is clear which measure $\mu$ or series $\sum_{n=1}^\infty x_n$ is considered. By $R(f_\mu)$ (or $R(f)$) we denote the range of the cardinal function $f_\mu$, that is the set of all cardinals $\vert\mu^{-1}(t)\vert$ where $t\in\text{rng}(\mu)$. Hence $R(f)\subset\{1,2,\dots,\omega,\mathfrak{c}\}$. For example, $n\in R(f_\mu)$ means that there are $n$ distinct sets $A_1,\dots,A_n\subset\N$ such that $\mu(A_i)=\mu(A_1)$ for $i=2,\dots,n$ and $\mu(A)\neq\mu(A_1)$ for any $A\in \mathcal{P}(\N)\setminus\{A_1,\dots,A_n\}$, where $\mathcal{P}(\N)$ stands for the family of all subsets of $\N$. If $f_\mu(t)=1$, then we say that $t$ is uniquely obtained, (or it has unique representation or unique expansion). The set $f_\mu^{-1}(1)$ will be called a set of uniqueness for $\mu$. By $Fin$ we denote the collection of all finite subsets of $\N$. Identifying subset of $\N$ with its characteristic functions, we consider on $\mathcal{P}(\N)$ the topology from the Cantor space $\{0,1\}^\N$. It turns out that $Fin$ is dense in $\mathcal{P}(\N)$, and therefore the set $\{\mu(A):A\in Fin\}$ is dense in $\on{rng}(\mu)$. We will use the notation $[1,n]:=\{1,2,\dots,n\}$. A subset of a complete metric space is called a Cantor set if it is non-empty compact and dense-in-itself. 

We are interested in the following problems: 
\begin{itemize}
\item for which subsets $R$ of $\{1,2,\dots,\omega,\mathfrak{c}\}$ there is a measure $\mu$ such that $R(f_\mu)=R$?
\item what are topological and set-theoretical properties of $f^{-1}_\mu(\alpha)$ for $\alpha\in\{1,2,\dots,\omega,\mathfrak{c}\}$? 
\end{itemize}
This kind of problems was considered by Bielas, Plewik and Walczy\'nska in \cite{BPW} for particular achievement set $\mathbf{T}$, prescribed below in Theorem \ref{char}, which is of a great importance in this field. The Authors showed that its cardinal function $f$ have only two values -- 1 and 2. Moreover $f^{-1}(i)$ has the cardinality continuum and is dense in $\mathbf{T}$ for $i=1,2$.

Achievement sets of sequences have been considered by many
authors; some results have been rediscovered several times. Let us list basic properties of $A(x_{n})$ (some of them discovered by Kakeya in \cite{Kakeya} in 1914):
\begin{theorem}\label{Kakeya}
Let $(x_n)$ be an absolutely summable sequence of real numbers. Then 
\begin{itemize}
\item[(i)] $A(x_n)$ is a compact perfect or finite set;

\item[(ii)] if $\vert x_{n}|>\sum_{i>n}|x_{i}|$ for all sufficiently large $n$'s,
then $A(x_{n})$ is homeomorphic to the Cantor set $C$;

\item[(iii)] if $|x_{n}|\leq \sum_{i>n}|x_{i}|$ for all sufficiently large $%
n $'s, then $A(x_{n})$ is a finite union of closed intervals. Moreover, if $%
|x_{n}|\geq |x_{n+1}|$ for all but finitely many $n$'s and $A(x_{n})$ is a
finite union of closed intervals, then $|x_{n}|\leq \sum_{i>n}|x_{i}|$ for
all but finitely many $n$'s.
\end{itemize}
\end{theorem}

One can see that $A(x_{n})$ is finite if and only if $x_{n}=0$ for all but
finite number of $n$'s, i.e. $(x_{n})\in c_{00}$. Kakeya conjectured that if 
$(x_{n})\in \ell _{1}\setminus c_{00}$, then $A(x_{n})$ is always
homeomorphic to the Cantor set $C$ or it is a finite union of intervals.
 In \cite{GN} Guthrie and Nymann gave a very simple
example of a sequence whose achievement set is not a finite union of closed
intervals but it has nonempty interior. They used the following sequence $%
(t_{n})=(\frac{3}{4},\frac{2}{4},\frac{3}{16},\frac{2}{16},\ldots )$.
Moreover, they formulated the following:

\begin{theorem}
\label{char} For any $(x_n) \in \ell_1 \setminus c_{00}$, the set $A(x_n)$
is one of the following types:

\begin{itemize}
\item[(i)] a finite union of closed intervals,

\item[(ii)] homeomorphic to the Cantor set $C$,

\item[(iii)] homeomorphic to the set $\mathbf{T}=A(t_n)=A(\frac{3}{4}, \frac{%
2}{4}, \frac{3}{16}, \frac{2}{16}, \frac{3}{64}, \ldots)$.
\end{itemize}
\end{theorem}

Although their proof had a gap, the theorem is true and the correct proof
was given by Nymann and Saenz in \cite{NS0}. Guthrie, Nymann and Saenz have
observed that the set $\mathbf{T}$ is homeomorphic to the set $N$ described
by the formula 
\begin{equation*}
N=[0,1]\setminus \bigcup_{n\in \mathbb{N}}U_{2n},
\end{equation*}%
where $U_{n}$ denotes the union of $2^{n-1}$ open middle thirds which are
removed from the interval $[0,1]$ at the $n$-th step in the construction of
the classic Cantor ternary set $C$. Such sets are called Cantorvals in the
literature (to emphasize the similarity to the interval and to the Cantor
set simultaneously). It is known that a Cantorval is such nonempty compact
set in $\mathbb{R}$, that it is the closure of its interior and both
endpoints of any nontrivial component are accumulation points of its trivial
components. Other topological characterizations of Cantorvals can be found
in \cite{BFPW} and \cite{MO}. The set $\mathbf{T}$ is called the Guthrie-Nymann Cantorval. 
All known examples of sequences which achievement sets are
Cantorvals belong to the class of multigeometric sequences. This class was
deeply investigated in \cite{Jones}, \cite{BFS} and \cite{BBFS}. In particular,
the achievement sets of multigeometric series and sets obtained in more
general case are the attractors of the affine iterated function systems, see 
\cite{BBFS}. More information on achievement sets can be found in the
surveys \cite{BFPW}, \cite{N1} and \cite{N2}.

Section \ref{omega} was inspired by the Franciszek Prus-Wi\'sniowski's question from 2017: is it possible to obtain some points of achievement set of absolutely convergent series for precisely countably many subseries? Or in our notation -- is it possible to construct an absolutely convergent series such that $f^{-1}(\omega)\neq\emptyset$? We have found that this problem had been already solved by Erd\"os, Horv\'ath and Jo\'o \cite{Erdos} in 1991, where the authors considered geometric series  $\sum_{n=1}^{\infty}q^n$. The authors study the question for which $q$ the number $1$ has exactly one, countably or continuum many expansions. We show that it is possible to construct a bi-geometric series such that there exist at least countably many points, obtained for exactly countably many expansions each. Indeed, combine Theorem \ref{resztywyrazy} (4) and Proposition \ref{geometrycznystronglyquickly}. Our construction differs from that \cite{Erdos}. 

The Erd\"os, Horv\'ath and Jo\'o  paper \cite{Erdos} initiated the large number of studies during  last 30 years on non-integer bases. In the most general approach by an expansion of a real number $x$ in base $q>1$ on the alphabet $\Sigma=\{0<\sigma_1<\dots<\sigma_n\}$ we mean a sequence $(c_n)$ of elements of $\Sigma$ satisfying the equality
\[
x=\frac{c_1}{q}+\frac{c_2}{q^2}+\dots
\]
If there is a sequence $(a_1,\dots,a_k)$ such that $\Sigma=\{\sum_{i=1}^k\varepsilon_i a_i:\varepsilon_i=0,1\}$, then the set of numbers $x\geq 0$ which have at least one expansion in base $q$ corresponds to the achievement set $A(a_1,\dots,a_k;q)$ of a multigeometric series $(a_1,\dots,a_k;q)$ given by
\[
a_1+\dots+a_k+a_1q+\dots+a_kq+a_1q^2+\dots+a_kq^2+\dots
\]
Given $q>1$ the set of all numbers $x$ whose expansion is unique in base $q$ on the alphabet $\Sigma$ is called \textit{univoque set}. The set, say $E$, of numbers which have at least one expansion is not necessarily equal to an achievement set of any series. To see this consider the alphabet $\Sigma=\{0,1,2,9\}$ and base $q=3$;  then $E$ is neither finite union of intervals nor Cantor set nor Cantorval, see \cite{NS1}.

It turns out that even if $\Sigma$ is of the form $\{\sum_{i=1}^k\varepsilon_i a_i:\varepsilon_i=0,1\}$ for some sequence $(a_1,\dots,a_k)$, the univoque set is not necessarily equal to $f^{-1}(\{1\})$ where $f$ stands for a cardinal function of a mutigeometric series $(a_1,\dots,a_k;q)$. For example consider a sequence $(1,1)$; then $\Sigma=\{0,1,2\}$ and for each unique expansion of a number $t$, on the alphabet $\Sigma$, which contains at least one digit 1, we have $f(t)>1$. However if $(a_1,\dots,a_k)$ is one-to-one and each $\sigma\in\Sigma$ has a unique representation $\sum_{i=1}^k\varepsilon_i a_i$ for some $\varepsilon_i=0,1$, the univoque set equals $f^{-1}(\{1\})$.  

This shows that the theory of non-integer bases and that we present are different.

\section{Measure $\mu$ with $f_\mu^{-1}(\omega)\neq\emptyset$}\label{omega}

A function $\chi:\mathcal{P}(\N)\to\R$ is called finite signed measure if it is countably additive, $\chi(\emptyset)=0$, $\sup\{\chi(A):A\subset\N\}<\infty$ and $\inf\{\chi(A):A\subset\N\}>-\infty$. If $x_n:=\chi(\{n\})$, then the series $\sum_{n=1}^\infty x_n$ is absolutely convergent. A measure $\vert\chi\vert(A)=\sum_{n\in A}\vert x_n\vert$ is called the variation of $\chi$. One can define a cardinal function $f_\chi$ for a signed measure $\chi$. 

Note that for every absolutely convergent series $\sum_{n=1}^{\infty}x_n$ we have $A(\vert x_n\vert)=A(x_n)+ \sum_{n=1}^{\infty}x_n^{-}$, where for a real number $x$ we consider $x^{-}:=\max\{-x,0\}$ and $x^{+}:=\max\{x,0\}$. Moreover, if $x=\sum_{n\in A}\vert x_n\vert$, then 
\[
x- \sum_{n=1}^{\infty}x_n^{-}=\sum_{n\in A} x_n^{+}+\sum_{n\in A} x_n^{-}-\sum_{n=1}^{\infty}x_n^{-}=\sum_{n\in A} x_n^{+}-\sum_{n\in \mathbb{N}\setminus A} x_n^{-}.
\]
Let $E=\{n : x_n>0\}$. Since $\psi: \mathcal{P}(\mathbb{N})\mapsto  \mathcal{P}(\mathbb{N})$ defined as $\psi (A):=(E\cap A)\cup((\mathbb{N}\setminus A)\cap(\mathbb{N}\setminus E))$ is bijective, we get $f(x-\sum_{n=1}^{\infty}x_n^{-})=g(x)$, where $f$ and $g$ are the cardinal functions for $\sum_{n=1}^{\infty}x_n$ and $\sum_{n=1}^{\infty}\vert x_n\vert$ respectively. This shows that $f_\chi$ and $f_{\vert\chi\vert}$ are equal up to shift of the domain. Therefore we will consider only measures or series with positive terms.

We start by giving some basic properties of cardinal function $f$. 
\begin{proposition}\label{podstawy}
Let $x_n=\mu(\{n\})$ for some measure $\mu$. Then
\begin{itemize}
\item[(i)] there are st least two uniquely obtained values in $\on{rng}(\mu)$;
\item[(ii)] $f(x)=f(\sum_{n=1}^{\infty}x_n-x)$ for each $x\in \mathbb{R}$; 
\end{itemize}
\begin{proof}
$(i)$. Clearly $\sum_{n\in A}x_n=\sum_{n=1}^{\infty}x_n$ if and only if $A=\mathbb{N}$. Since $x_n>0$, $\sum_{n\in B}x_n=0$ if and only if $B=\emptyset$. Hence $0$ and $\sum_{n=1}^{\infty}x_n$ have unique representations.

$(ii)$. Let $\mathcal{A}$ be the family of all sets $A$ such that $\sum_{n\in A}x_n=x$. Then $\mathcal{A}^c=\{\mathbb{N}\setminus A : A\in\mathcal{A}\}$ is the family of all sets  $\mathbb{N}\setminus A$ such that $\sum_{n\in \mathbb{N}\setminus A}x_n=\sum_{n=1}^{\infty}x_n-x$. Since $A\mapsto\mathbb{N}\setminus A$ is bijection, we obtain $\vert \mathcal{A}\vert=\vert \mathcal{A}^c\vert$.
\end{proof}
\end{proposition}

Before we answer the Prus-Wi\'sniowski's question, we give some definitions. For a convergent series $\sum_{n=1}^{\infty}x_n$ of positive terms, by $r_n=\sum_{k=n+1}^{\infty} x_k$ we denote its $n$-th remainder. We say that $\sum_{n=1}^{\infty}x_n$ is quickly convergent if $x_n>r_n$  for each $n\in\mathbb{N}$. A series $\sum_{n=1}^{\infty}x_n$ has property 
\begin{itemize}
\item[(A)] if $x_i>\sum_{n=i+1}^{\infty} (x_n+r_{n-1})$ for each $i\in\mathbb{N}$;
\item[(B)] if $r_{i-1}>\sum_{n=i+1}^{\infty} (x_n+r_{n-1})$ for each $i\in\mathbb{N}$.
\end{itemize}

Since $r_{i-1}>x_i$, it is easy to see that each (A) series is a (B) series. Furthermore every (B) series is quickly convergent. Indeed, if $r_{i-1}>\sum_{n=i+1}^{\infty} (x_n+r_{n-1})$, then from $r_{i-1}=x_i+r_{i}$ we get $x_i>\sum_{n=i+1}^{\infty} (x_n+r_{n})>\sum_{n=i+1}^{\infty}x_n$.
\begin{theorem}\label{resztywyrazy}
Let   $\sum_{n=1}^{\infty}x_n$ be a (B) series. We define $y_{2n-1}=r_{n-1}$, $y_{2n}=x_n$ for each $n\in\mathbb{N}$. Let $f$ be the cardinal function of $\sum_{n=1}^{\infty}y_n$. Then:
\begin{enumerate}
\item $f(r_n)=\omega$ for every $n\in\N$;
\\Moreover if $\sum_{n=1}^{\infty}x_n$ is (A) series then:
\item $f(x_n)=1$ for every $n\in\N$;
\item $f( \sum_{n\in F}x_n)=1$ for each $F\subset\mathbb{N}$ such that $\vert\mathbb{N}\setminus F\vert=\infty$;
\item $f(\sum_{n\in F}x_n)=\omega$ for each $F\subset\mathbb{N}$ such that $1\leq\vert\mathbb{N}\setminus F\vert<\infty$;
\item there are continuum many points $t$ such that $f(t)=\mathfrak{c}$;
\item $R(f)=\{1,2,\ldots,k,\ldots,\omega,\mathfrak{c}\}$.
\end{enumerate}
\end{theorem}

\begin{proof}
(1) Let $W(t):=\{A\subset\N:\sum_{n\in A}y_n=t\}$. Fix $k\in\mathbb{N}_0:=\N\cup\{0\}$. Then $r_k=x_{k+1}+r_{k+1}=x_{k+1}+x_{k+2}+r_{k+2}=\ldots=\sum_{n=k+1}^{m}x_n+r_{m}$ for each $m\geq k$. Thus $W(r_k)\supset\{\{2k+2,2k+4,\ldots,2k+2m,2k+2m+1\}: m\in\mathbb{N}_0 \}$. Hence $f(r_k)\geq \omega$. 
Note that if $A\in W(r_k)$ then $A\subset[2k+1,\infty)$, since for $n\leq 2k$ we have $y_n>r_k$. 
We will prove inductively that $W(r_k)=\bigcup_{m\in\mathbb{N}_0}W^k_m$, where $W^k_m=\{2k+2,2k+4,\ldots,2k+2m,2k+2m+1\}$ for each  $m\in\mathbb{N}_0$. Namely we will prove that if $A\in W(r_k)$ and $\vert A\vert=m+1$, then $A=W_m$. For $m=0$ and $m=1$ the assertion is obvious. Assume that we have already proved that for each $k\in\mathbb{N}_0$ and for some $p\in\mathbb{N}$ every $p$-elements set $A\in W(r_k)$ equals to $W^k_{p-1}$. Let $B\in W(r_k)$, $\vert B\vert=p+1$. Note that  $2k+1\in B$ implies $\vert B\vert=1$, which yields a contradiction. Thus $2k+1\notin B$. Suppose that $2k+2\notin B$. Then $r_k>\sum_{n=k+2}^{\infty}(x_n+r_{n-1})>\sum_{n\in B}y_n=r_k$ and again we get a contradiction. Then $2k+2\in B$, so $\sum_{n\in B\setminus\{2k+2\}}y_n=\sum_{n\in B}y_n-y_{2k+2}=r_k-x_{k+1}=r_{k+1}$. Since $\vert  B\setminus\{2k+2\}\vert=p$ by the inductive assumption we obtain $B\setminus\{2k+2\}=\{2k+4,2k+6,\ldots,2k+2p,2k+2p+1\}=W^{k+1}_{p-1}$. Thus $B=\{2k+2,2k+4,2k+6,\ldots,2k+2p,2k+2p+1\}$.

To finish the proof of the first part we need to show that if $A\in W(r_k)$ is infinite, then $A=2\mathbb{N}\setminus\{2,4,\ldots,2k\}$. Note that $2k+1\notin A$ and $2k+2\in A$. Assume that we have already proved that $2k+2i-1\notin A$ and $2k+2i\in A$ for each $i\leq p$. We need to show that $2k+2p+1\notin A$ and $2k+2p+2\in A$. Since $\sum_{i=1}^{p}y_{2k+2i} + y_{2k+2p+1}=r_k$ and $\vert A\vert >p+1$ we obtain that $2k+2p+1\notin A$. Moreover $r_k-\sum_{i=1}^{p}y_{2k+2i} =r_{k+p}>\sum_{i\geq 2k+2p+3}y_i$, so $2k+2p+2\in A$.
By the induction we obtain $A=2\mathbb{N}\setminus\{2,4,\ldots,2k\}$. Thus for each $k\in\mathbb{N}$ we have $f(r_k)=\omega$.

(2) Fix $k\in\mathbb{N}$ and assume that $\sum_{n\in A}y_n=x_k$. Note that for $n<2k$ we have $y_n>x_k$, so $A\cap[1,2k-1]=\emptyset$. If $2k\notin A$, the inequality $x_k>\sum_{n>2k}y_n$ leads to a contradiction. Hence $A=\{2k\}$. Thus $f(x_k)=1$.

(3) Let $F=(f_n)_{n=1}^{k}$ be an increasing sequence of indices with  $\vert\mathbb{N}\setminus F\vert=\infty$, where $k\in\mathbb{N}\cup\{\infty\}$ and $x:=\sum_{n=1}^{k}x_{f_n}=\sum_{n=1}^{k}y_{2f_n}$. Let $G\subset\N$ be such that $x=\sum_{n\in G}y_n$. Since $\vert\mathbb{N}\setminus F\vert=\infty$, then $\sum_{n\geq i}x_{f_n}<r_{f_i-1}$. Therefore $y_{2f_1-1}>x$ and consequently $G\cap[1,2f_1)=\emptyset$. Note that $y_{2f_1}\in G$. Otherwise 
\[
x=\sum_{n\in G}y_n<\sum_{n>2f_1}y_n=\sum_{n=f_1+1}^\infty x_n+\sum_{n=f_1+1}^\infty r_{n-1}<x_{f_1}=y_{2f_1}\leq x
\]
which yields a contradiction. Next $x-y_{2f_1}<r_{f_2-1}$, so $G\cap(2f_1,2f_2)=\emptyset$. As before $y_{2f_2}\in G$; otherwise $x-y_{2f_1}<\sum_{n>2f_2}y_n<y_{2f_2}\leq x-y_{2f_1}$. By a simple induction we obtain $2f_n\in G$ for each $n\leq k$. Since $y=\sum_{n=1}^{k}y_{2f_n}$ we obtain $G=2F$. Hence $f(x)=1$.

(4)  Now let $F=(f_n)_{n=1}^{\infty}$ be an increasing sequence of indices and  $\vert\mathbb{N}\setminus F\vert<\infty$. Thus $F=(f_n)_{n=1}^{m}\cup[p,\infty)$ for some $p>f_m+1$ and $m\in\mathbb{N}$. Let $x=\sum_{n\in F}x_n=\sum_{n=1}^{m}x_n+r_{p-1}$. Assume that  $x=\sum_{n\in G}x_n$. Using the same argument as before we show that $G\supset (2f_n)_{n=1}^{m}$. Thus $\sum_{n\in G\setminus(2f_n)_{n=1}^{m}}y_n=r_{p-1}$. By $(1)$ there exists infinite and countable family of sets $\mathcal{H}=(H_k)_{k=1}^{\infty}$ such that $r_{p-1}=\sum_{n\in H_k}y_n$ for each $k\in\mathbb{N}$. Let $\mathcal{G}=(G_k)_{k=1}^{\infty}$, where $G_k=H_k\cup (2f_n)_{n=1}^{m}$ for every $k\in\mathbb{N}$, and such that $r_{p-1}=\sum_{n\in H}y_n$ implies $H=H_k$ for some $k\in\N$. Clearly $G$ is an element of $\mathcal{G}$. Note that $x=\sum_{n\in G_k}y_n$ for every $k\in\mathbb{N}$. Thus $\vert\mathcal{G}\vert=\omega$ and we obtain the assertion.

(5) We will show that $f(x_H)=\mathfrak{c}$ for $x_H:=\sum_{n\in H}x_{2n}+\sum_{n=1}^{\infty}r_{2n}$ where $H\subset\N$. Since $r_{2n}=x_{2n+1}+r_{2n+1}$ for each $n\in\mathbb{N}$, we obtain $x_H=\sum_{n\in H}x_{2n}+\sum_{n=1}^{\infty}v_{n}$, where by every $v_n$ we can take $r_{2n}$ either $x_{2n+1}+r_{2n+1}$. 

(6) Fix $k\in\mathbb{N}$ . We will show that $f(x)=k$ for $x:=\sum_{n=1}^{\infty}x_{kn}+r_{0}$. Since 
$r_0=x_1+r_1=x_1+x_2+r_2=\ldots=x_1+x_2+\ldots+x_{n-1}+r_{n-1}$, it is clear that $f(x)\geq k$. Moreover by (A) series property, if we ommit $r_0$ in our sum, we must take $x_1$, if we ommit $r_1$, we must take $x_2$, and so on. Therefore each representation $\sum_{n\in M}x_n$ of $x$ need to contain $r_0$ or $r_1+x_1$ or one of the $k-2$ remaining representations of $r_0$. Now let us assume that $x=\sum_{n\in M} x_{n}+r_{0}$. Note that $p\notin M$ for each $p<k$, since otherwise $\sum x_{m_n}+r_0\geq x_p+r_0>x$. On the other hand if $k\notin M$, then $\sum_{n\in M} x_{n}+r_0<x_k+r_0<x$ and we reach a contradiction as well. Hence $k\in M$. In the same way for consequtive $m\in\mathbb{N}$, we inductively eliminates the elements $x_p$ for $p\in [mk+1,(m+1)k)$ and show that $(m+1)k\in M$. Hence $M=\{kn:n\in\N\}$ and there are no new subseries which sum up to $x$. That is $f(x)=k$. 
\end{proof}

Now we give the characterization of (A) series and (B) series within the class of geometric series.
\begin{proposition}\label{geometrycznystronglyquickly}
Let us consider a geometric sequence $x_n:=q^n$. Then the series $\sum_{n=1}^{\infty}x_n$ has
\begin{itemize}
\item[(i)] property  (A) for $q\in (0,\frac{2-\sqrt{2}}{2})$;
\item[(ii)] property (B) for $q\in (0,\frac{3-\sqrt{5}}{2})$.
\end{itemize}
\begin{proof}
(i) Fix $k\in\mathbb{N}$. We have $r_{m-1}=\sum_{n=m}^{\infty}q^n=\frac{q^m}{1-q}$ for every  $m\in\mathbb{N}$. The series has (A) property if and only if $q^k=x_k>\sum_{n=k+1}^{\infty} (x_n+r_{n-1})=\sum_{n=k+1}^{\infty} (q^n+\frac{q^n}{1-q})=\frac{(2-q)q^{k+1}}{(1-q)^2}$. Thus 
$$
1>\frac{(2-q)q}{(1-q)^2}\Leftrightarrow 1-2q+q^2>2q-q^2\Leftrightarrow 2q^2-4q+1>0\Leftrightarrow q\in\left(-\infty, \frac{2-\sqrt{2}}{2}\right)\cup \left(\frac{2+\sqrt{2}}{2},\infty\right).
$$
Hence the geometric series has (A) property for $q\in (0,\frac{2-\sqrt{2}}{2})$.

(ii) The series has (B) property if and only if $r_{k-1}=\frac{q^k}{1-q}>\frac{(2-q)q^{k+1}}{(1-q)^2}$. Hence 
$$
1>\frac{(2-q)q}{(1-q)}\Leftrightarrow 1-q>2q-q^2\Leftrightarrow q^2-3q+1>0\Leftrightarrow  q\in\left(-\infty, \frac{3-\sqrt{5}}{2}\right)\cup \left(\frac{3+\sqrt{5}}{2},\infty\right).
$$
It shows that the geometric series has (B) property for $q\in (0,\frac{3-\sqrt{5}}{2})$.
\end{proof}
\end{proposition}
\begin{remark}\emph{
Let $x_n=q^n$ for some $q\in(\frac{2-\sqrt{2}}{2},\frac{3-\sqrt{5}}{2})$, that is by the Proposition \ref{geometrycznystronglyquickly} the series $\sum_{n=1}^{\infty}x_n$ is (B), but it is not (A). We define the terms of the series $\sum_{n=1}^{\infty}y_n$ in the same way as in the Theorem \ref{resztywyrazy}. Fix $n\in\mathbb{N}$. We have: 
$$y_{2n-1}=r_{n-1}=\sum_{k=n}^{\infty}x_k=\sum_{k=n}^{\infty}y_{2k}\leq \sum_{k=2n}^{\infty}y_{k}$$
and 
$$y_{2n}=x_n\leq\sum_{k=n}^{\infty} (x_{k+1}+r_{k})=\sum_{k=2n+1}^{\infty}y_{k},$$
since $\sum_{n=1}^{\infty}x_n$ is not (A) series. Hence $y_{n}\leq \sum_{k=n+1}^{\infty}y_{k}$ for each $n\in\mathbb{N}$. Then $A(y_n)=[0,\sum_{n=1}^{\infty}y_n]$. By Theorem \ref{resztywyrazy} we get that $f^{-1}(\{\omega\})\neq\emptyset$. Hence it is possible to construct a measure, which range is a closed interval and there exists a value, which is obtained for exactly  countably infinite many distinct sets.}
\end{remark}


Proposition \ref{geometrycznystronglyquickly} show us that we can construct a series in Thoerem \ref{resztywyrazy} in the way it is bigeometric. It is a subclass of so-called multigeometric series. A multigeometric series is each series, whose terms are the following: 
$$
(k_0,k_1,k_2,\ldots,k_m,k_0q,k_1q,k_2q,\ldots,k_mq,k_0q^2,k_1q^2,k_2q^2,\ldots,k_mq^2,\ldots)
$$
for some finite set $\{k_0,k_1,k_2,\ldots,k_m\}$ of real numbers and $q\in (0,1)$. A bigeometric series is a multigeometric series with $m=1$. Using Thoerem \ref{resztywyrazy} for $x_n=q^n$ we obtain the bigeometric series $\sum_{n=1}^{\infty}y_n$, where $k_0=\frac{q}{1-q}$, $k_1=q$. Moreover, by using the methods of Theorem \ref{resztywyrazy} we cannot obtain $(y_n)$ to be a geometric sequence; by a simple calculations $(y_n)$ is geometric provided $q=\frac{3-\sqrt{5}}{2}$, which leads to a series, which does not have property (B).

\begin{proposition}\label{wszystkobezomegi}
There exists a measure $\mu$ for which $R(f_\mu)=\mathbb{N}\cup\{\mathfrak{c}\}$, that is the cardinal function $f_\mu$ has all possible values but $\omega$.  
\end{proposition}
\begin{proof}
Let $A_n=[\frac{n^2-n+2}{2},\frac{n^2+n}{2}]$ for each $n\in\mathbb{N}$. Clearly $A_1,A_2,\dots$ is a partition of $\N$ with $\vert A_n\vert=n$. Let $\sum_{n=1}^{\infty}x_n$ be any series satisfying the following two conditions:
\begin{itemize}
\item[(i)] for each $n\in\mathbb{N}$ and any $j,l\in A_n$ we have $x_j=x_l$;
\item[(ii)] for each $n\in\mathbb{N}$ and $k\in A_n$  the inequality $x_k>\sum_{m>n}\sum_{i\in A_m}x_i$ holds.
\end{itemize}
For example $x_k=\frac{1}{10^{n^2}}$ for $k\in A_n$ and $n\in\mathbb{N}$ satisfies conditions (i) and (ii). Note that $\sum_{i\in B_j}x_i=\sum_{i\in C_j}x_i$, where $B_j\subset A_j$, $C_j\subset A_j$ for $j\in\mathbb{N}$, implies, by (i) and (ii), that $\vert B_j\vert=\vert C_j\vert$ for all $j$. Let $x=\sum_{i\in B_j}x_i$. Firstly, consider the case when $\emptyset\neq B_j\neq A_j$ holds for finitely many $j$'s, and denote them by $j_1,\ldots, j_r$.   Then $f(x)=\prod_{p=1}^{r} {j_p\choose\vert B_{j_p}\vert}$. In particular for an element $x$ such that $r=1$, $\vert B_{j_1}\vert=1$, $j_1=n$ we obtain $f(x)=n$. Hence $\mathbb{N}\subset R(f)$. Secondly, let us assume that  $\emptyset\neq B_j\neq A_j$ holds for infinitely many $j$'s. Then clearly $f(x)=\mathfrak{c}$.
\end{proof}

The next example shows that there exists a measure $\mu$ such that $R(f_\mu)=\{1,\mathfrak{c}\}$.
\begin{example}\label{skrajne}
\emph{Let $x_{2n-1}=x_{2n}=\frac{1}{2^n}$. Then $A(x_n)=[0,2]$. By Proposition \ref{podstawy} we know that $f(0)=f(2)=1$. Assume that $x\in (0,1]$. Then there exists unique, infinite subset $F=(k_n)_{n=1}^{\infty}$ of odd idices such that  $x=\sum_{n\in F}x_n$. Let us consider $\mathcal{F}=\{(k_n+j_n)_{n=1}^{\infty} : (j_n)\in\{0,1\}^{\mathbb{N}}\}$. Since for every $A\in\mathcal{F}$ we have $x=\sum_{n\in A}x_n$ and $\vert\mathcal{F}\vert=\mathfrak{c}$. 
From Proposition \ref{podstawy}(2) we obtain $f(x)=\mathfrak{c}$ for every  $x\in(0,2)$.}
\end{example}
One can prove a more general statement. Assume that for an absolutely convergent series $\sum_{n=1}^{\infty}x_n$ there exists two sequences of pairwise disjoint sets of indices $(F_n),(G_n)$ such that for all $i,j\in\mathbb{N}$ we have $F_i\cap G_j=\emptyset$ and $\sum_{n\in G_i}x_n=\sum_{n\in F_i}x_n$.
Then $f^{-1}(\mathfrak{c})$ contains a Cantor set. Indeed, if $t=\sum_{i=1}^\infty\sum_{n\in F_i}x_n$, then 
\[
t=\sum_{i=1}^\infty\sum_{n\in F^{\varepsilon_i}_i}x_n,
\] 
where $(\varepsilon_i)\in\{0,1\}^\N$, $F_i^0=F_i$ and $F_i^1=G_i$. Therefore $f(t)=\mathfrak{c}$. The same holds for $t_X=\sum_{i\in X}^\infty\sum_{n\in F_i}x_n$ for every infinite set $X\subset\N$. Let $t_i=\sum_{n\in F_i}x_n$. Then
\[
A(t_i)=\{t_X:X\text{ is infinite}\}\cup\{t_X:X\text{ is finite}\}
\]
Note that $\{t_X:X\text{ is finite}\}$ is countable and $A(t_i)$ is uncountable compact. Therefore $f^{-1}(\mathfrak{c})$ contains an uncountable $G_\delta$ set $\{t_X:X\text{ is infinite}\}$; in particular $f^{-1}(\mathfrak{c})$ contains a Cantor set \cite{Kechris}. 

Every example of a measure $\mu$ with $f_\mu^{-1}(\mathfrak{c})\neq\emptyset$ we have examined had the property that $f_\mu^{-1}(\mathfrak{c})$ contains a Cantor set. However we are not able to determine whether such phenomena holds for every measure. 
\begin{problem}
Let $\sum_{n=1}^{\infty}x_n$ be an absolutely convergent series and $f^{-1}(\mathfrak{c})\neq\emptyset$. Is it true that  $f^{-1}(\mathfrak{c})$ contains a Cantor set?
\end{problem}

In general, if there is value $t\in\on{rng}(\mu)$ which is achieved infinitely many times, then there are infinitely many such values. More precisely the following holds. 
\begin{proposition}
Assume that there is $t\in\on{rng}(\mu)$ with $f_{\mu}(t)\geq\omega$. Then the set $f_\mu^{-1}(\{\omega,\mathfrak{c}\})$ is dense in $\on{rng}(\mu)$. 
\end{proposition}

\begin{proof}
Assume that $f(t)=\omega$, that is $\vert \{A\subset\mathbb{N} : \sum_{n\in A}x_n=t\}\vert=\omega$. Hence at least one set of the form $\{A\subset\mathbb{N} : 1\in A, \sum_{n\in A}x_n=t\}$ or $\{A\subset\mathbb{N} : 1\notin A, \sum_{n\in A}x_n=t\}$ is of cardinality $\omega$. Inductively we define a 0-1 sequence $\alpha\in\{0,1\}^\N$ such that for any $k$ the family
$$
\Big\{A\subset\N:\sum_{n\in A}x_n=t\text{ and }\forall n\leq k(x_n\in A\iff\alpha(n)=1)\Big\}
$$
has cardinality $\omega$. Let $\{n_1<n_2<\dots\}$ be an increasing enumeration of $\{i\in\N:\alpha(i)=1\}$. Then for any $k$ 
$$
\Big\vert\Big\{A\subset\mathbb{N} : n_1, n_2, \ldots, n_k\in A \text{ and } \sum_{n\in A}x_n=t\Big\}\Big\vert=\omega\iff
$$
$$
\Big\vert\Big\{A\subset\mathbb{N}\setminus\{n_1,\dots,n_k\} : \sum_{n\in A}x_n=t-\sum_{i=1}^{k}x_{n_i}\Big\}\Big\vert=\omega\iff f\Big(t-\sum_{i=1}^{k}x_{n_i}\Big)=\omega.
$$
Let $(\varepsilon_i)_{i=1}^m\in\{0,1\}^m$. There is $k$ such that $n_k\geq m$ and $t-\sum_{i=1}^{k}x_{n_i}<\varepsilon$. For any $A\subset\N\setminus\{n_1,\dots,n_k\}$ with $\sum_{n\in A}x_n=t-\sum_{i=1}^{k}x_{n_i}$ the element $\sum_{i=1}^m\varepsilon_ix_i+\sum_{n\in A}x_n$ is $\varepsilon$-closed to $\sum_{i=1}^m\varepsilon_ix_i$ and $f(\sum_{i=1}^m\varepsilon_ix_i+\sum_{n\in A}x_n)\geq\omega$. Since the set of elements of the form $\sum_{i=1}^m\varepsilon_ix_i$ is dense in $\text{rng}(\mu)$, we obtain the assertion. Clearly the same argument works if $f(t)=\mathfrak{c}$. 
\end{proof}

We do not know if $f(t)=\omega$ implies that there is $t'$ with $f(t')=\mathfrak{c}$. We also do not know if $f^{-1}(\omega)\neq\emptyset$ implies that $f^{-1}(\omega)$ is dense in $\on{rng}(\mu)$. 
\begin{problem}
Let $\sum_{n=1}^{\infty}x_n$ be an absolutely convergent series and $f^{-1}(\omega)\neq\emptyset$. Is it true that $f^{-1}(\mathfrak{c})\neq\emptyset$? Is it true that $f^{-1}(\omega)$ is dense in $\on{rng}(\mu)$?
\end{problem}

In paper \cite{Erdos}, the authors considered geometric series for which some value is obtained by exactly $\omega$ subseries. All but one of the subseries was finite. In the next Propositon we show that infinite subseries always appears.
\begin{proposition}
Let $\mu$ be a measure. If $f(t)=\omega$, then there exists an infinite $A\subset\mathbb{N}$ with $\mu(A)=t$.
\end{proposition}

\begin{proof}
Let $F:\{0,1\}^\N\to\R$ be given by $F(\varepsilon_n)=\sum_{n=1}^\infty\varepsilon_n x_n$ where $x_n:=\mu(\{n\})$. As we have mentioned in the Introduction the function $F$ is continuous. Therefore $F^{-1}(t)$ is closed in the compact Cantor space $\{0,1\}^{\mathbb{N}}$. Since $F^{-1}(t)$ is infinite compact, it has an accumulation point $A$, that is there is a one-to-one sequence $(A_n)\subset\{0,1\}^\N$ such that $F(A_n)=t$ and $A_n$ tends to $A$ in $\{0,1\}^\N$. If there exists $n$ with $\vert A_n\vert=\omega$, then we are done. Otherwise $\vert A_n\vert<\omega$ for every $n\in\N$. We will show that $\vert A\vert=\omega$. Note that there are $2^m$ subsets of $\N$ with cardinality $\leq m$. Since the sequence $(A_n)$ is one-to-one, all but finitely many $A_n$'s have cardinality $>m$. Passing to a subsequence we may assume that $\vert A_n\vert\geq n$ for every $n\in\N$. For a fixed $m$ there is $k$ such that $A\cap[1,m]=A_n\cap[1,m]$ for every $n\geq k$. Suppose that there is $m$ with $A\subset[1,m]$. We can find $n>m$ such that $A=A_n\cap[1,m]$. Then there is $l>m$ with $l\in A_n$. We have
$$
t=\mu(A)=\mu(A_n\cap[1,m])<\mu(A_n\cap[1,m])+\mu(\{l\})\leq\mu(A_n)=t
$$
which yields a contradiction. Hence $A$ is infinite.
\end{proof}

By Proposition \ref{podstawy} we know that $\mu(\emptyset)$ and $\mu(\N)$ are uniquely obtained. As shows the Example \ref{skrajne} there is a measure $\mu$ with $\mu(\emptyset)$ and $\mu(\N)$ are the only such values. 
Sequences $(x_n)$ for which $A(x_n)$ is an interval (so called interval filling sequence) and its set of uniqueness consists of $\{0,\sum x_n\}$ only, have already been studied in \cite{DJK}, \cite{DK} and \cite{DKS}; such sequences are called plentiful in  \cite{DKS}. The authors also considered a sequence $(x_n)$ (or a series $\sum_{n=1}^{\infty}x_n$) which is a locker, that is $A((x_n)_{n\neq N})$ is an interval for each $N\in\mathbb{N}$.  In \cite{DJK} it is proved that each locker is plentiful. Note that the series defined in the Example \ref{skrajne} is a locker, so that it is also plentiful. Using simple observations about quasiregular expansions from \cite{LS} we are able to show some deeper idea staying behind Example \ref{skrajne}. An expansion $x=\sum_{n=1}^{\infty}\varepsilon_n x_n$ is quasiregular if and only if $(\varepsilon_n)$ is defined in the following way: having defined $\varepsilon_i$ for $i\in[1,n-1]$ we put $\varepsilon_n=1$ if and only if $\sum_{k=1}^{n-1}\varepsilon_k x_k +x_n<x$. The authors of \cite{DJK} observed that for every interval filling sequence of positive terms $(x_n)$ every $x\in (0,\sum_{n=1}^{\infty}x_n]$ has an infinite quasiregular expansion.  It gives a universal way of transforming any interval filling sequence into  locker with all but two points obtained for $\mathfrak{c}$-many expansions.
\begin{corollary}
Let $(x_n)$ be an interval filling, positive sequence. Then $(y_n)$ defined as $y_{2n}=y_{2n-1}=x_n$ for each $n\in\mathbb{N}$ is locker and for the cardinal function of $(y_n)$ we have $f(y)=\mathfrak{c}$ for all $y\in (0,2\sum_{n=1}^{\infty}x_n)$.
\end{corollary}

\begin{problem}
Is it true that the existence of a sequence $(n_k)\to\infty$ such that  $f^{-1}(n_k)\neq\emptyset$ for all $k\in\mathbb{N}$ implies that $f^{-1}(\{\omega,\mathfrak{c}\})\neq\emptyset$ ? Note that by Proposition \ref{wszystkobezomegi} we know that it may happen that $f^{-1}(\omega)=\emptyset$.
\end{problem}

\begin{problem}
Does $f^{-1}(\omega)\neq\emptyset$ (or the conjunction $f^{-1}(\omega)\neq\emptyset$  and $f^{-1}(\mathfrak{c})\neq\emptyset$) implies the existence of a sequence $(n_k)\to\infty$ such that  $f^{-1}(n_k)\neq\emptyset$ for all $k\in\mathbb{N}$. Note that by Example \ref{skrajne} the condition $f^{-1}(\mathfrak{c})\neq\emptyset$ does not imply the assertion.
\end{problem}

\section{Kakeya like condition}

In Theorem \ref{Kakeya}, attributed to Kakeya, the properties of $\on{rng}(\mu)$($=A(x_n)$) depend on the relation between terms $x_n=\mu(\{n\})$ and remainders $r_n=\sum_{k>n}x_k=\mu((n,\infty))$. We show that the quick convergence, that is $x_n>r_n$ for every $n$, implies that each point of $\on{rng}(\mu)$ is uniquely obtained; slow convergence, that is $x_n\leq r_n$ for all but finitely many $n$'s, implies that the set of uniqueness has an empty interior while $\on{rng}(\mu)$ is a union of finitely many closed intervals.

Let us start from prescribing a class of measures for which every values is obtained in the unique way. 

\begin{proposition}\label{szybkozbiezne}
Let $\mu$ be a measure and let $x_n:=\mu(\{n\})$. If a series $\sum_{n=1}^\infty x_n$ is quickly convergent, i.e. $\mu(\{n\})>\mu(\{n+1,n+2,\dots\})$ for every $n\in\N$, then each point of $\on{rng}(\mu)$ has unique representation. 
\begin{proof}
Suppose that $x=\sum_{n=1}^\infty\varepsilon_nx_n=\sum_{n=1}^\infty\delta_nx_n$ for two distinct sequences $(\varepsilon_{n})_{n=1}^{\infty}\in\{0,1\}^{\mathbb{N}}$ and $(\delta_{n})_{n=1}^{\infty}\in\{0,1\}^{\mathbb{N}}$. Let $k$ be the smallest natural number such that $\varepsilon_{k}\neq\delta_k$. Without losing generality let $\varepsilon_{k}=1$ and $\delta_k=0$. Then $0=x-x=\sum_{n=1}^\infty\varepsilon_nx_n-\sum_{n=1}^\infty\delta_nx_n=\sum_{n=1}^\infty(\varepsilon_n-\delta_n)x_n=x_k-\sum_{n=k+1}^\infty(\varepsilon_n-\delta_n)x_n\geq x_k-\sum_{n=k+1}^\infty x_n>0$, which yields a contradiction.
\end{proof}
\end{proposition}
As shows the next example, there exists a series $\sum_{n=1}^\infty x_n$, which is not quickly convergent, but the assertion of Proposition \ref{szybkozbiezne} remains true. It means that one cannot reverse Proposition \ref{szybkozbiezne}.
\begin{example}\label{nieszybkozbiezny}\emph{
Let $x_1:=\frac{27}{32}$ and $x_n:=\frac{3}{4^{n-1}}$ for $n\geq 2$. The series is not quickly convergent, since $x_1=\frac{27}{32}<1=\sum_{n=2}^\infty x_n$. On the other hand the series  $\sum_{n=2}^\infty x_n$ is quickly convergent so for $y_n=x_{n+1}$ every point of the achievement set $A(y_n)$ is obtained in the unique way. Note that $A(y_n)\subset [0,\frac{2}{32}]\cup [\frac{6}{32},\frac{8}{32}]\cup [\frac{24}{32},\frac{26}{32}]\cup [\frac{30}{32},\frac{32}{32}]$. Thus $x_1+A(y_n)\subset [\frac{27}{32},\frac{29}{32}]\cup [\frac{33}{32},\frac{35}{32}]\cup [\frac{51}{32},\frac{53}{32}]\cup [\frac{57}{32},\frac{59}{32}]$, so $A(y_n)\cap(x_1+A(y_n))=\emptyset$ (see the picture). Hence $f^{-1}(\{1\})=A(y_n)\cup (x_1+A(y_n))=A(x_n)$.}

\end{example}
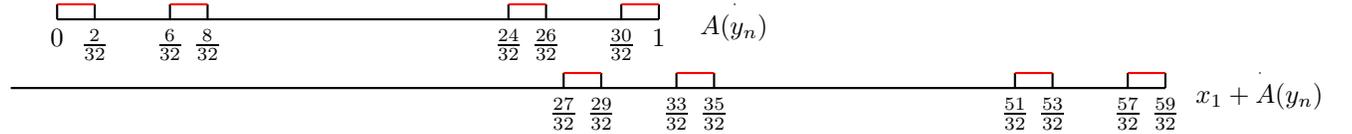
\begin{tikzpicture}[>=stealth,thick] 
 \draw[-] (0,0)--(8,0);
 \draw (0,0.2)--(0,0) node[below]{$0$};
 \draw (0.5,0.2)--(0.5,0) node[below]{$\frac{2}{32}$};
 \draw (1.5,0.2)--(1.5,0) node[below]{$\frac{6}{32}$};
 \draw (2,0.2)--(2,0) node[below]{$\frac{8}{32}$};
 \draw (6,0.2)--(6,0) node[below]{$\frac{24}{32}$};
 \draw (6.5,0.2)--(6.5,0) node[below]{$\frac{26}{32}$};
 \draw (7.5,0.2)--(7.5,0) node[below]{$\frac{30}{32}$};
 \draw (8,0.2)--(8,0) node[below]{$1$};
 \draw (9,0.2)--(9,0.2) node[below]{$A(y_n)$};
\draw[-] (0,0.2)--(0.5,0.2)[color=red];
\draw[-] (1.5,0.2)--(2,0.2)[color=red];
\draw[-] (6,0.2)--(6.5,0.2)[color=red];
\draw[-] (7.5,0.2)--(8,0.2)[color=red];
\end{tikzpicture}
\\
\begin{tikzpicture}[>=stealth,thick]
 \draw[-] (-0.1,0)--(15.25,0);
 \draw (7.25,0.2)--(7.25,0) node[below]{$\frac{27}{32}$};
 \draw (7.75,0.2)--(7.75,0) node[below]{$\frac{29}{32}$};
 \draw (8.75,0.2)--(8.75,0) node[below]{$\frac{33}{32}$};
 \draw (9.25,0.2)--(9.25,0) node[below]{$\frac{35}{32}$};
 \draw (13.25,0.2)--(13.25,0) node[below]{$\frac{51}{32}$};
 \draw (13.75,0.2)--(13.75,0) node[below]{$\frac{53}{32}$};
 \draw (14.75,0.2)--(14.75,0) node[below]{$\frac{57}{32}$};
 \draw (15.25,0.2)--(15.25,0) node[below]{$\frac{59}{32}$};
\draw (16.5,0.2)--(16.5,0.2) node[below]{$x_1+A(y_n)$};
\draw[-] (7.25,0.2)--(7.75,0.2)[color=red];
\draw[-] (8.75,0.2)--(9.25,0.2)[color=red];
\draw[-] (13.25,0.2)--(13.75,0.2)[color=red];
\draw[-] (14.75,0.2)--(15.25,0.2)[color=red];
\end{tikzpicture}

Now we consider the case when $\on{rng}(\mu)$ is a closed interval or a finite union of closed intervals. It turns out that the set of uniqueness is topologically small. 
\begin{theorem}\label{przedzial}
Let $\mu$ be a measure and let $x_n:=\mu(\{n\})$. Assume that $\on{rng}(\mu)$ is a closed interval. Then $f_\mu(x_{k_1}+x_{k_2}+\dots+x_{k_m})\geq 2$ for every finite set $k_1<k_2<\dots<k_m$ of indexes. In particular the set of uniqueness of $\mu$ has an empty interior.
\end{theorem}

\begin{proof}
Firstly we prove that $f(x_n)\geq 2$ for every $n\in\N$. Since $\on{rng}(\mu)$ is an interval, the sequence $(x_n)$ is slowly convergent. Fix $n\in\N$. If $x_n=r_n=\sum_{k=n+1}^\infty x_n$, then we are done. So we assume that $x_n<r_n$. Inductively we will define consecutive segments $F_1,F_2,\dots\subset[n+1,\infty)$ such that $\sum_{i=1}^\infty\sum_{k\in F_i}x_k=x_n$ as follows. 

Since $x_n<r_n$, there is $l>n$ with $\sum_{k=n+1}^l x_k\leq x_n<\sum_{k=n+1}^{l+1}x_k.$ Put $F_1:=[n+1,l]$. If $\sum_{k\in F_1}x_k=x_n$, then put $F_i:=\emptyset$ for $i\geq 2$ and we are done. Otherwise let $u:=\min\{i>F_1:\sum_{k\in F_1}x_k+x_i<x_n\}$. Note that $x_{u-1}\geq x_n-\sum_{k\in F_1}x_k>0$ and $x_{u-1}\leq r_{u-1}$. If $x_u=x_n-\sum_{k\in F_1}x_k$ and $x_{u-1}=r_{u-1}$, then $x_n=\sum_{k\in F_1}+r_{u-1}$ and we are done (put for example $F_j:=\{j+u-2\}$ for $j\geq 2$). Otherwise we continue the construction; there is $v\geq u$ such that
$$
\sum_{k\in F_1}x_k+\sum_{k=u}^v x_k\leq x_n<\sum_{k\in F_1}x_k+\sum_{k=u}^{v+1} x_k.
$$
Put $F_2:=[u,v]$. Now we proceed the construction as in the first step. By the construction we have $\vert\sum_{i=1}^k\sum_{j\in F_i}x_j-x_k\vert\leq x_{\max F_k+1}$. Therefore $\sum_{i=1}^\infty\sum_{k\in F_i}x_k=x_n$. 

Thus $f(x_n)\geq 2$. Now we show that $\sum_{k\in F}x_k$ has a unique representation provided $F$ is finite. Fix $F=\{k_1<k_2<\ldots<k_m\}$. By the first part of the proof we know that $x_{k_m}=\sum_{k\in A}x_k$ for some $A\subset(k_m,\infty)$. Hence $\sum_{k\in F}x_k=\sum_{i=1}^{m}x_{k_i}=\sum_{i=1}^{m-1}x_{k_i}+\sum_{k\in A}x_k$. Thus $f(\sum_{k\in F}x_k)\geq 2$. Note that $\{\sum_{k\in F}x_k : F\in Fin\}$ is a dense subset of $A(x_n)$. Hence the set of uniqueness has an empty interior.
\end{proof}

\begin{corollary}\label{sumofintervals}
Assume that $\on{rng}(\mu)$ is a finite union of closed interval, equivalently $x_n\leq r_n$ for all but finitely many $n$'s. Then the set of uniqueness has an empty interior.
\end{corollary}

\begin{proof}
Let $x_n:=\mu(\{n\})$. By Theorem \ref{Kakeya} there exists $m\in\mathbb{N}$ such that $x_n\leq\sum_{k=n+1}^{\infty}x_{k}$ for each $n\geq m$. By Theorem \ref{przedzial} used to the series $\sum_{n=m}^\infty x_n$ we obtain $f(\sum_{k\in F}x_k)\geq 2$ provided $F$ is finite and $F\cap[m,\infty)\neq\emptyset$. Since $\{\sum_{k\in F}x_k : F\in Fin, F\cap[m,\infty)\neq\emptyset \}$ is dense in $\on{rng}(\mu)$, we obtain that the set of uniqueness has an empty interior.
\end{proof}

Although the set of uniqueness has an empty interior, it may be quite large as shows the following.

\begin{example}
\emph{
Let us consider a sequence $(y_n)$ given by $y_n:=\frac{1}{2^n}$. Then $A(y_n)=[0,1]$. Note that 
$$
f(t)=2\iff\exists m\exists(\varepsilon_i)_{i=1}^m\in\{0,1\}^m\Big(\varepsilon_m=1\text{ and } t=\sum_{i=1}^m\frac{\varepsilon_i}{2^i}=\sum_{i=1}^{m-1}\frac{\varepsilon_i}{2^i}+\sum_{i=m+1}^\infty\frac{1}{2^i}\Big)\iff
$$
$$
\iff t\in(0,1)\text{ has finite dyadic expansion.}
$$
Since other points in $(0,1)$ have unique dyadic expansions, the set of uniqueness is cocountable, and therefore it is comeager and conull in $[0,1]$. 
}
\end{example}

\begin{lemma}\label{tailequalterm} 
Let $\mu$ be a measure such that $\mu(\{n\})=\mu((n,\infty))$ for every $n\in\N$. Then\\ 
(i) $R(f_\mu)=\{1,2\}$;\\ 
(ii) $f_\mu(y)=2$ if and only if $y=\sum_{n\in A}x_n$ for some finite set $A$. 
\end{lemma}

\begin{proof}
Fix $A$ such that $\vert A\vert<\infty$. Let $A:=\{n_i: i\in [1,k]\}$ and consider $y:=\mu(A)$. As always put $x_n:=\mu(\{n\})$ and $r_n:=\mu((n,\infty))$. Then 
$$
y=x_{n_1}+x_{n_2}+ \ldots +x_{n_{k-1}}+x_{n_k}=x_{n_1}+x_{n_2}+ \ldots +x_{n_{k-1}}+r_{n_k}=x_{n_1}+x_{n_2}+ \ldots +x_{n_{k-1}}+\sum_{n=n_k+1}x_{n},
$$ 
and consequently $f(y)\geq 2$. 

Suppose that $f(y)>2$. Then there would be $C\subset\N$ such that $\mu(A)=\mu(B)=\mu(C)$, $C\neq A$ and $C\neq B$ where $B=(A\setminus\{n_k\})\cup(n_k,\infty)$. Let $C=\{t_i:i\in I\}$ where $I=\N$ or $I=[1,p]$ for some $p\in\N$. If $t_1<n_1$, then 
$$
\sum_{n\in C}x_n\geq x_{t_1}\geq \sum_{i>t_1}x_i>\sum_{n\in A}x_n
$$
which is a contradiction (the latter inequality holds if $k>1$, otherwise we can say only that $\sum_{i>t_1}x_i\geq\sum_{n\in A}x_n$; it is possible that $\sum_{i>t_1}x_i=\sum_{n\in A}x_n=x_{n_1}$, but then $C=[n_1+1,\infty)$; this means that $C=B$ and we obtain a contradiction as well). Similarly $t_1>n_1$ yields a contradiction. Thus $t_1=n_1$. Inductively one can show that $t_1=n_1,\dots,t_{k-1}=n_{k-1}$ and either $t_k=n_k$ or $t_{k+i}=n_k+i+1$ for every $i\geq 0$. Hence $C=A$ or $C=B$. This is a contradiction. Note that $f(\mu(A))=2$ also if $A$ is a confinite set of indexes. 

Let $A\subset\N$ and put $y:=\mu(A)$. We will show that $f(A)\geq 2$ implies that $A$ is either finite or cofinite. Assume that there exists $B\neq A$ such that  $\sum_{n\in A}x_n=\sum_{n\in B}x_n$. Denote $k=\min(A\Delta B)$ where $A\Delta B$ stands for the symmetric difference $(A\setminus B)\cup(B\setminus A)$ of sets $A$ and $B$. Hence 
$$
0=\Big\vert \sum_{n\in A}x_n-\sum_{n\in B}x_n\Big\vert\geq x_k- \Big\vert \sum_{n\in A_{k+1}}x_n-\sum_{n\in B_{k+1}}x_n\Big\vert\geq x_k-r_k=0,
$$
which means that $\vert\sum_{n\in A_{k+1}}x_n-\sum_{n\in B_{k+1}}x_n\vert=r_k$, so $A_{k+1}=\mathbb{N}\setminus [1,k]$ and $B_{k+1}=\emptyset$ or vice-versa. But if $A_{k+1}$ is cofinite, so is $A$, and if $A_{k+1}$ is empty, then $A$ is finite. 
\end{proof}

\begin{lemma}\label{boundedconditions}
Let $\sum_{n=1}^{\infty}x_{n}$ be a convergent series of positive terms and let $f$ be its cardinal function. Moreover let $f_k$ stand for the cardinal function of $\sum_{n=k}^{\infty}x_{n}$. The following assertions are equivalent: 
\begin{itemize}
\item[(i)] $f$ is bounded;
\item[(ii)] $f_k$ is bounded for all $k\in\mathbb{N}$;
\item[(iii)] $f_k$ is bounded for some $k\in\mathbb{N}$.
\end{itemize}
\end{lemma} 

\begin{proof}
$(i)\Rightarrow (ii)$. Since $A(x_n)\supset A((x_n)_{n\geq k})$ and $f_k\leq f$ on $A((x_n)_{n\geq k})$ we are done.
\\$(ii)\Rightarrow (iii)$. Obvious.
\\$(iii)\Rightarrow (i)$. Assume that $f_k$ is bounded. By the already proved implication $(i)\Rightarrow (ii)$ it is enough to show that $f$ is bounded. First, we show that $f_{k-1}$ is bounded. Fix $y\in A((x_n)_{n\geq k})$. Note that 
$$
\Big\{A\subset[k-1,\infty): \sum_{n\in A}x_n=y\Big\}=\Big\{A\subset [k,\infty): \sum_{n\in A}x_n=y\Big\}\cup \Big\{A\subset [k,\infty): \sum_{n\in A}x_n=y-x_{k-1}\Big\}
$$ 
is finite as a union of two finite sets. Hence $f_{k-1}$ is bounded. Using a simple induction, we show that $f_{k-2},\ldots,f_2,f_1=f$ are bounded. 
\end{proof}

\begin{theorem}\label{OnKakeyaCondition}
Let $\mu$ be a measure with $\mu(\{n\})\leq\mu((n,\infty))$. Then 
\begin{itemize}
\item[(i)] if $\mu(\{n\})<\mu((n,\infty))$ for finitely many $n\in\mathbb{N}$, then $f_\mu$ is bounded; 
\item[(ii)] if $\mu(\{n\})<\mu((n,\infty))$ for infinitely many $n\in\mathbb{N}$, then $f_\mu^{-1}(\{\omega,\mathfrak{c}\})\neq\emptyset$; 
\end{itemize}
\end{theorem}

\begin{proof} Let $x_n:=\mu(\{n\})$ and $r_n:=\mu((n,\infty))$. 
(i). Let $k$ be the largest natural number $n$ for which the inequality  $x_n<r_n$ holds. Therefore $x_n=r_n$ for $n\geq k+1$. By Lemma \ref{tailequalterm} the cardinal function $f_{k+1}$ is bounded. Hence by Lemma \ref{boundedconditions} we obtain the boundedness of $f$.

(ii). Firstly we inductively define a sequence of indexes $k_0<k_1<\dots$ such that 
\begin{equation}\label{tails}
x_{k_j}<r_{k_j}\text{ for every }j\geq 0\text{ and }\sum_{j=l}^ix_{k_j}<r_{k_l}\text{ for every }0\leq l\leq i.
\end{equation}
Find $k_0$ such that $x_{k_0}<r_{k_0}$. Then \eqref{tails} holds for $l=i=0$. Assume that we have already constructed $k_0<k_1<\dots<k_{p-1}$. We find $k_p>k_{p-1}$ such that $x_{k_p}<r_{k_p}$ and $x_{k_p}<\min_{l<p}(r_{k_l}-\sum_{j=l}^{p-1}x_{k_j})$. Then \eqref{tails} holds for all $0\leq l\leq i\leq p$.  

Now, consider an element $t:=x_{k_0}+x_{k_1}+x_{k_2}+\dots$. Note that
$$
t-\sum_{j=0}^{l-1}x_{k_j}=\sum_{j=l}^\infty x_{k_j}\leq r_{k_l}.
$$
Thus $t-\sum_{j=0}^{l-1}x_{k_j}\in[0,r_{k_l}]$ and the slow convergence implies that $[0,r_{k_l}]=A((x_n)_{n>k_l})$. Therefore there is $A_l\subset[k_l+1,\infty)\cap\N$ such that 
$$
t-\sum_{j=0}^{l-1}x_{k_j}=\sum_{n\in A_l}x_n.
$$

Hence $t=\sum_{n\in B_l}x_n$ where $B_l=\{k_0,k_1,\dots,k_{l-1}\}\cup A_l$. Since $k_0,k_1,\dots,k_{l-1}\in B_l$ and $k_l\notin B_l$, then $B_1,B_2,\dots$ are pairwise distinct, and therefore $f(t)\geq\omega$.
\end{proof}

Immediately we obtain the following. 
\begin{corollary}\label{CorollaryOnKakeyaCondition}
Let $\mu$ be a measure such that $f_\mu^{-1}(\{\omega,\mathfrak{c}\})=\emptyset$ and $f_\mu$ is unbounded in $\N$. Then $\mu(\{n\})>\mu((n,\infty))$ for infinitely many $n$'s. 
\end{corollary}

We do not know if there is a measure which fulfills the assumptions of Corollary \ref{CorollaryOnKakeyaCondition}.  

\begin{remark}\label{RemarkGuthrieNeymannCantorval}
\emph{
Note that the fact that inequality $x_n<r_n$ holds for infinitely many $n$'s does not imply that $f^{-1}(\{\omega,\mathfrak{c}\})\neq\emptyset$. Consider the following sequence $(x_n):=(\frac{2}{4},\frac{3}{4},\frac{2}{4^2},\frac{3}{4^2},\dots)$. As we have mentioned in the Introduction the set $\mathbf{T}=A(x_n)$ is known as the Guthrie-Neymann Cantorval. Bielas, Plewik and Walczy\'nska showed in \cite{BPW} that $R(f)=\{1,2\}$ where $f$ is the cardinal function of $A(x_n)$. This shows that the slow convergence assumption in Theorem \ref{OnKakeyaCondition} is essential. 
}
\end{remark}

By a gap in $E$ we understand any such interval $(a,b)$ that $a,b\in E$ and $(a,b)\cap E=\emptyset $. The following lemma can be found in \cite{BFGPWS}. 
\begin{lemma}\label{gaplemma}
Suppose that $(a,b)$ is a gap in $E=A(x_{n})$ such that for
any gap $\left( a_{1},b_{1}\right) $ with $b_{1}<a$ we have $b-a>b_{1}-a_{1}$
(in other words $\left( a,b\right) $ is the longest gap from the left). Then 
$b=x_{k}$ for some $k\in \mathbb{N}$ and $a=r_{k}$.
\end{lemma}

\begin{proposition}
Let $\mu$ be a measure which set of uniqueness equals $\{0,\mu(\N)\}$. Then $\on{rng}(\mu)$ is an interval. 
\end{proposition}

\begin{proof}
Suppose to the contrary that $\on{rng}(\mu)$ is not an interval. Then is contains a gap. Take the longest gap from the left in $\on{rng}(\mu)$, say $(a,b)$. By Lemma \ref{gaplemma} $a$ and $b$ have unique representations. This yields a contradiction. 
\end{proof}

\begin{remark}\emph{
The examples considered so far may suggest that the set of uniqueness is either small (meager) in $A(x_n)$ or large (comeager) in $A(x_n)$. The following simple example shows that this is not the case. Consider the sequence $x_1=x_2:=1/2$ and $x_n:=1/2^{n-1}$ for $n\geq 3$. Then $A(x)=[0,3/2]$; all but countably many points in $[0,1/2)$ and in $(1,3/2]$ have unique expansion, while all points in $[1/2,1]$ can be obtained for at least two ways (for the expansion of any $x\in[1/2,1]$ one need to use $x_1$ or $x_2$). 
}
\end{remark}

Let us go back to the Guthrie-Nymann Cantorval $\mathbf{T}$. As we have mentioned in the Introduction and in Remark \ref{RemarkGuthrieNeymannCantorval} Bielas, Plewik and Walczy\'nska in \cite{BPW} determined the cardinal function for $\mathbf{T}$. They have characterized the set of all points of the Cantorval, which are not obtained in the unique way and have shown that such points have exactly two expansions. It turns out that the set of uniqueness is dense in $\mathbf{T}$, has cardinality continuum, but it does not contain any interval (its complement in $\mathbf{T}$ is dense as well). 

\begin{problem}
Does there exists a measure $\mu$ such that its set of uniqueness contains an interval? Note that $\on{rng}(\mu)$ cannot be a finite union of intervals by Corollary \ref{sumofintervals}. 
\end{problem}


Note that the properties of the set of uniqueness strongly depends on the direction of the inequality between elements $x_n$ and the tails $r_n$. Namely if the series $\sum_{n=1}^{\infty}x_{n}$ is quickly convergent, that is $x_n>r_n$ for each $n\in\mathbb{N}$, then each point of the measure range is obtained uniquely by Proposition \ref{szybkozbiezne}. Example \ref{nieszybkozbiezny} shows that the latter equality may also hold for series, which is not quickly convergent. On the other hand the equality $x_n\leq r_n$ holds for large enough $n$ if and only if $A(x_n)$ is a finite union of closed intervals. In that case the set of uniqueness has an empty interior in $\on{rng}(\mu)$, by Corollary \ref{sumofintervals}. In the next Example we show that there exists a series for which $x_n\leq r_n$ holds for infinitely many $n$, but each point of the measure range is obtained uniquely. It strengthens the Example  \ref{nieszybkozbiezny} and shows that the assumption in  Corollary \ref{sumofintervals} is optimal in sense of relation between $x_n$'s and remainders $r_n$'s.

\begin{example}\emph{
We define $x_{2n-1}:=\frac{1}{10^{n^2}}+\frac{1}{2\cdot 10^{(n+1)^2}}$, $x_{2n}:=\frac{1}{10^{n^2}}$ for all natural $n$. Note that $x_{2n-1}\leq x_{2n}+x_{2n+1}<r_{2n-1}$ for all $n\in\mathbb{N}$. We will show that each point of $A(x_n)$ has unique representation. Here we give a list of properties of the series $\sum_{n=1}^{\infty}x_{n}$ that we will use in the sequel; their proofs are tedious and we will left them to the readers. All of the inequalities holds for all $k\in\mathbb{N}$.
\begin{itemize}
\item[(i)] $x_{2k}>r_{2k}$;
\item[(ii)] $x_{2k+1}>r_{2k+2}$;
\item[(iii)] $x_{2k+1}-x_{2k+2}>r_{2k+5}$;
\item[(iv)] $x_{2k+2}+x_{2k+4}-x_{2k+1}>r_{2k+4}$;
\item[(v)] $ x_{2k+1}+x_{2k+3}-x_{2k+2}-x_{2k+4}>r_{2k+4}$;
\item[(vi)] $x_{2k+1}+x_{2k+4}-x_{2k+2}-x_{2k+3}>r_{2k+4}$.
\end{itemize}
Let us assume that there exists $x\in A(x_n)$ such that $f(x)\geq 2$, that is $x=\sum_{n\in A_0}x_n+\sum_{n\in A_1}x_n=\sum_{n\in B_0}x_n+\sum_{n\in B_1}x_n$ for some $A_0,B_0\subset 2\mathbb{N}$, $A_1,B_1\subset 2\mathbb{N}-1$, where $A_0\neq B_0$ or $A_1\neq B_1$. Note that both $(x_{2n})_{n=1}^\infty$ and $(x_{2n-1})_{n=1}^\infty$ are quickly convergent. So if $A_i=B_i$, then $\sum_{n\in A_{1-i}}x_n=\sum_{n\in B_{1-i}}x_n$, and by the quick convergence of the sequence $(x_{2n+1-i})_{n=1}^\infty$, we obtain $A_{1-i}=B_{1-i}$ which is impossible. Thus $A_0\neq B_0$ \textbf{and} $A_1\neq B_1$. 
\\Let $2k+1$ be the minimal index belonging to $A_1\setminus B_1$. By (i) we obtain $A_0\cap\{1,\ldots,2k\}=B_0\cap\{1,\ldots,2k\}$. By (ii) we have $2k+2\in B_0\setminus A_0$. By (iii), we get $2k+4\in B_0\setminus A_0$ or $2k+3\in B_1\setminus A_1$. 
If $2k+4\in B_0\setminus A_0$, then by the inequality (iv), we get $2k+3\in A_1\setminus B_1$. If $2k+3\in B_1\setminus A_1$, then we have $x_{2k+2}+x_{2k+3}-x_{2k+1}>r_{2k+4}$, so  $2k+4\in A_0\setminus B_0$. Therefore:
\begin{center}
\textbf{either} $2k+4\in B_0\setminus A_0$ and $2k+3\in A_1\setminus B_1$ \textbf{or} $2k+3\in B_0\setminus A_0$ and $2k+4\in A_1\setminus B_1$. 
\end{center} 
For shortening a notation for any set $A$ we denote $A^{k}=A\cap\{1,\ldots,k\}$.
Hence
$$
0=\Bigg\vert \sum_{n\in A_0}x_n+\sum_{n\in A_1}x_n-\sum_{n\in B_0}x_n-\sum_{n\in B_1}x_n\Bigg\vert=
$$
$$
\Bigg\vert\sum_{n\in A_0^{2k+4}}x_n+\sum_{n\in A_1^{2k+4}}x_n-\sum_{n\in B_0^{2k+4}}x_n-\sum_{n\in B_1^{2k+4}}x_n+\sum_{n\in A_0\setminus A_0^{2k+4}}x_n+\sum_{n\in A_1\setminus A_1^{2k+4}}x_n-\sum_{n\in B_0\setminus B_0^{2k+4}}x_n-\sum_{n\in B_1\setminus B_1^{2k+4}}x_n\Bigg\vert\geq
$$
$$
\Bigg\vert\sum_{n\in A_0^{2k+4}}x_n+\sum_{n\in A_1^{2k+4}}x_n-\sum_{n\in B_0^{2k+4}}x_n-\sum_{n\in B_1^{2k+4}}x_n\Bigg\vert-\Bigg\vert\sum_{n\in A_0\setminus A_0^{2k+4}}x_n+\sum_{n\in A_1\setminus A_1^{2k+4}}x_n-\sum_{n\in B_0\setminus B_0^{2k+4}}x_n-\sum_{n\in B_1\setminus B_1^{2k+4}}x_n\Bigg\vert
$$
$$ 
\geq\Bigg\vert\sum_{n\in A_0^{2k+4}}x_n+\sum_{n\in A_1^{2k+4}}x_n-\sum_{n\in B_0^{2k+4}}x_n-\sum_{n\in B_1^{2k+4}}x_n\Bigg\vert-r_{2k+4}>0.
$$ 
The last strict inequality, which leads to a contradiction is followed by (v) in case (A), and by (vi) when (B).
}
\end{example}

\section{Measures with prescribed cardinal functions}

In this section we show how to produce a measure with given range of its cardinal function. We also give examples of sets $R\subset\N$ such that there is no measure $\mu$ with $R=R(f_\mu)$. Let us start from the case when the set of uniqueness is the whole range $\on{rng}(\mu)$ up to some finite set. 
\begin{proposition}\label{niematrojkiwsrodku}
Assume that the set of uniqueness for $\mu$ equals $\on{rng}(\mu)\setminus\{x\}$ for some $x\in\on{rng}(\mu)$. Then $f_\mu(x)=2$ and there are $A,B\subset\N$ with $\mu(A)=\mu(B)$, $A\cap B=\emptyset$ and $A\cup B=\N$. 
\end{proposition}
\begin{proof} 
Suppose to the contrary that $f(x)\geq 3$. Since $f$ is symmetric it is clear that $x=\frac{1}{2}\mu(\N)$. Otherwise we would have $f(x)\geq 3$ and $f(\sum_{n=1}^{\infty}x_n-x)\geq 3$ by Proposition \ref{podstawy}(2). There exist three distinct sets $A,B,C$ such that $x=\mu(A)=\mu(B)=\mu(C)$. Note that the sets $A,B,C$ cannot be pairwise disjoint, since then $\mu(A\cup B\cup C)=\mu(A)+\mu(B)+\mu(C)=3x>\mu(\N)$. Without loss of generality we may assume that $A\cap B\neq\emptyset$. Hence for $y:=\mu( A\setminus B)=\mu(B\setminus A)<x$ we have $f(y)\geq 2$. But this contradicts the assumption.

Therefore $f(x)=2$. Hence there are $A,B\subset\N$ with $\mu(A)=\mu(B)=\frac{1}{2}\mu(\N)$. As before we show that $A\cap B=\emptyset$. Since $\mu(A\cup B)=\mu(\N)$ and the support of $\mu$ equals $\N$, we obtain that $A\cup B=\N$.  
\end{proof}

The following example shows that the assumption of Proposition \ref{niematrojkiwsrodku} is fulfilled for some measure. 

\begin{example}\label{dwojkawsrodku}\emph{
Let  $\sum_{n=1}^{\infty}y_n$ be any absolutely convergent series such that each of its point has unique representation. Let us define $x_1:=\sum_{n=1}^{\infty}y_n$ and $x_{n+1}:=y_n$ for $n\in\mathbb{N}$. Then $A(x_n)=A(y_n)\cup (A(y_n)+x_1)$ and $A(y_n)\cap (A(y_n)+x_1)=\{x_1\}$. Consequently the set of uniqueness equals $A(x_n)\setminus\{x_1\}$ and $f(x_1)=2$. Here the sets from the assertion of Proposition \ref{niematrojkiwsrodku} are $A:=[2,\infty)$ and $B:=\{1\}$.
}
\end{example}

We can also construct an example in which both sets $A$ and $B$ are infinite.

\begin{example}\label{nieskonczonedwojkawsrodku}\emph{
Let $x_1:=\frac{8}{9}$, $x_2:=\frac{2}{3}$, $x_{2n+1}:=\frac{1}{10^n}$,  $x_{2n+2}:=\frac{3}{10^n}$ for $n\in\mathbb{N}$. Note that $\sum_{n=1}^{\infty}x_n=2$ and $\sum_{n=1}^{\infty}x_{2n-1}=\sum_{n=1}^{\infty}x_{2n}=1$. The sets from the assertion of Proposition \ref{niematrojkiwsrodku} are $A:=2\N$ and $B:=2\N-1$. We will show that other points of $\on{rng}(\mu)$ are uniquely obtained where $\mu$ is the measure given by the series $\sum_{n=1}^\infty x_n$. Let $x\in\on{rng}(\mu)$. There is $E\subset\N$ with $x=\mu(E)$. Using the fact that $0\leq x\leq\mu(\N)=2$ and the symmetry of $\on{rng}(\mu)$, it is enough to prove the uniqueness of $E$ for $x<1$. Consider the following cases.
\\ \textbf{Case 1.} $0<x<x_2$. Then $1,2\notin E$. Thus $x\in \{\sum_{n=1}^{\infty}\frac{\delta_n}{10^n} : (\delta_n)\in\{0,1,3,4\}^{\omega}\}$. For every $n\in\mathbb{N}$ we have 
\begin{itemize}
\item $\delta_n=0$ if and only if $2n+1\notin E$ and $2n+2\notin E$;
\item $\delta_n=1$ if and only if $2n+1\in E$ and $2n+2\notin E$;
\item $\delta_n=3$ if and only if $2n+1\notin E$ and $2n+2\in E$;
\item $\delta_n=4$ if and only if $2n+1\in E$ and $2n+2\in E$.
\end{itemize}
Hence the set $E$ is a unique representation of $x$. 
\\ \textbf{Case 2.} $x_2<x<x_1$. Since $\mu([3,\infty))<\mu(\{2\})=x_2<x<x_1=\mu(\{1\})$, then $2\in E$ and $1\notin E$. Note that $x-x_2=\mu(E)\setminus \mu(\{2\})=\mu(E\setminus\{2\})$. Since $x-x_2<x_1-x_2<x_2$, then by already proved Case 1 we obtain that $E\setminus\{2\}$ is a unique representation of $x-x_2$. Since every representation of $x$ contains $x_2$, the set $E$ is unique. 
\\ \textbf{Case 3.} $x_1<x\leq 1$. Assume that there is $D\subset\N$, $D\neq E$ with $x=\mu(D)$. We will show that $x=1$. If $1\in D\cap E$ then $x-x_1<1-x_1<x_2$ and by the already proved Case 1 the representation $x-x_1=\mu(E\setminus \{1\})$ is unique. Thus $D\setminus\{1\}=E\setminus\{1\}$, and consequently $D=E$ which is a contradiction. Hence $1\notin D\cap E$. In the same way one can prove that $2\notin D\cap E$. Since $x>\mu([3,\infty))$, then $1\in D$ and $2\in E$ or vice-versa. Let assume that $1\in D$ and $2\in E$. We will prove inductively that $D$ and $E$ are the sets of all odd and even numbers, respectively. Let us assume that for some $n\in\mathbb{N}$ we have already prove that $D\supset\{1,3,\ldots,2n-1\}$, $D\cap\{2,4,\ldots,2n\}=\emptyset$ and  $E\supset\{2,4,\ldots,2n\}$, $E\cap\{1,3,\ldots,2n-1\}=\emptyset$. Since 
$$
\sum_{k=1}^{n}x_{2k}+x_{2n+1}+\sum_{k=2n+3}^{\infty}x_{k}=\Big[\frac{2}{3}+\frac{1}{3}\Big(1-\frac{1}{10^{n-1}}\Big)\Big]+\frac{1}{10^n}+\frac{4}{9\cdot 10^n}<\frac{8}{9}+\frac{1}{9}\Big(1-\frac{1}{10^{n-1}}\Big)=\sum_{k=1}^{n}x_{2k-1}\leq \mu(D),
$$ 
we obtain $2n+2\in E$. Note that $\sum_{k=1}^{n+1}x_{2k}+x_{2n+1}>1$, so $2n+1\notin E$. Since $\sum_{k=1}^{n}x_{2k-1} + x_{2n+2}>1$, we obtain $2n+2\notin D$. Moreover $\sum_{k=1}^{n}x_{2k-1}+\sum_{k=2n+3}^{\infty}x_{k}<\sum_{k=1}^{n+1}x_{2k}\leq\mu(E)$, so $2n+1\in D$. By the induction we obtain that $D=2\mathbb{N}-1$ and $E=2\mathbb{N}$. Thus $x=1$. Note that we have also shown that $\mu(C)=1$ if and only if $C=2\mathbb{N}-1$ or $C=2\mathbb{N}$. Hence $f(1)=2$.
\\ \textbf{Case 4.} $x=x_2$. It is clear that we must have $E=\{2\}$.
\\ \textbf{Case 5.} $x=x_1$. Suppose $1\notin E$. Since $x_1>\mu([3,\infty))$, then $2\in E$. Therefore $x-x_2\in\on{rng}(\mu)$. 
Note that 
$$
x-x_2=\frac{2}{9}=\sum_{n=1}^\infty\frac{2}{10^n}\notin \{\sum_{n=1}^{\infty}\frac{\delta_n}{10^n} : (\delta_n)\in\{0,1,3,4\}^{\omega}\}=\on{rng}(\mu)\cap[0,x_2).
$$ 
This is a contradiction. Thus $E=\{1\}$ is a unique representation of $x_1$. 
}
\end{example}

Now we will focus on the more general situation when the set of uniqueness of a measure $\mu$ equals $\on{rng}(\mu)$ up to some finite set $F$. The following shows that for any $k$ is is possible to define $\mu$ such that $F$ has exactly $k$ points. Let us start from $k=2$. 
\begin{example}\label{bezdwoch}\emph{
Let  $\sum_{n=1}^{\infty}y_n$ be any absolutely convergent series with positive terms such that $f_y^{-1}(\{1\})= A(y_n)$. Let us define $x_1:=3\sum_{n=1}^{\infty}y_n$, $x_2:=\sum_{n=1}^{\infty}y_n$ and $x_{n+2}:=y_n$ for each $n\in\mathbb{N}$. Clearly $A(x_n)=A(y_n)\cup (A(y_n)+x_2)\cup(A(y_n)+x_1)\cup(A(y_n)+x_1+x_2)$. Moreover $f(x_2)=f(x_1+x_2)=2$  and $f^{-1}(\{1\})=A(x_n)\setminus\{x_1,x_1+x_2\}$.}
\end{example}

For arbitrary $k$ we have the following. 
\begin{proposition}\label{skonczeniewielepodwojnych}
Fix $k\in\mathbb{N}$. Then there exists a measure $\mu$ such that its set of uniqueness equals $\on{rng}(\mu)$ up to a set $F$ with $\vert F\vert =k$, and such that $f_{\mu}(t)=2$ for $t\in F$.
\end{proposition}
\begin{proof}
We will proceed inductively. For $k=1$ it is constructed in Example \ref{dwojkawsrodku} (or in Example \ref{nieskonczonedwojkawsrodku}). Assume that $\sum_{n=1}^{\infty}y_n$ is a series such that $A(y_n)=f^{-1}(\{1\})\cup f^{-1}(\{2\})$ and $\vert f^{-1}(\{2\})\vert=k$, where $f$ is a cardinal function for $\sum_{n=1}^{\infty}y_n$. We will show that it is possible to construct two series  $\sum_{n=1}^{\infty}x^1_n$ and $\sum_{n=1}^{\infty}x^2_n$ such that $A(x^i_n)=f_i^{-1}(\{1\})\cup f_i^{-1}(\{2\})$ and:
\begin{enumerate}
\item $\vert f_1^{-1}(\{2\})\vert=2k$;
\item $\vert f_2^{-1}(\{2\})\vert=2k+1$
\end{enumerate}
where $f_i$ is a cardinal function for $\sum_{n=1}^{\infty}x^i_n$.\\
(1). Let $x^1_1$ be any real number larger than $\sum_{n=1}^{\infty} y_n$, $x^1_{n+1}:=y_n$ for $n\in\mathbb{N}$. Then $A(x^1_n)=A(y_n)\cup (x^1_1+A(y_n))$ and we have $f_1(x)=f(x)$ for all $x\in A(y_n)$ and $f_1(x)=f(x-x_1)$ for each $x\in x_1+A(y_n)$. 
\\(2). Let $x^2_1:=\sum_{n=1}^{\infty} y_n$, $x^2_{n+1}:=y_n$ for $n\in\mathbb{N}$. Then $A(x^2_n)=A(y_n)\cup (x^2_1+A(y_n))$. We have $f_2(x)=f(x)$ for all $x\in A(y_n)\setminus\{x^2_1\}$, $f(x^2_1)=2$ and $f_2(x)=f(x-x^2_1)$ for each $x\in (x^2_1+A(y_n))\setminus\{x^2_1\}$. 
\end{proof}

The construction in the proof of Proposition \ref{skonczeniewielepodwojnych} non-unique points are obtained in exactly two ways. Now we will prove that this is not a special situation but general phenomena. 
\begin{proposition}\label{trojkanieskonczona}
Let $\mu$ be a measure such that its set of uniqueness equals $\on{rng}(\mu)$ up to a finite set $F$. Then any value from $F$ has exactly two representations. In other word, if there is a value $x\in \on{rng}(\mu)$ with $f_\mu(x)\geq 3$, then there are infinitely many non-unique values in $\on{rng}(\mu)$. 
\end{proposition}
\begin{proof}
Suppose that there exist three distinct sets $A,B,C\subset\N$ with $\mu(A)=\mu(B)=\mu(C)$. Note that $\mu(A\setminus\{k\})=\mu(B\setminus\{k\})$ and $A\setminus\{k\}\neq B\setminus\{k\}$ for $k\in A\cap B$. Therefore $A\cap B$ has to be a finite set. Similarly $A\cap C$, $C\cap B$ are finite. Further, $\mu(A\cup\{k\})=\mu(B\cup\{k\})$ and $A\cup\{k\}\neq B\cup\{k\}$ for $k\in \N\setminus(A\cap B)$. Therefore the sets $A\cup B$, $A\cup C$, $B\cup C$ have to be cofinite. Hence $C\setminus (A\cup B)$ is finite as a subset of finite set $\mathbb{N}\setminus (A\cup B)$; similarly $B\setminus (A\cup C)$ is finite. Thus 
$$
B\cup C=[C\setminus(A\cup B)]\cup[B\setminus(A\cup C)]\cup(B\cap C)\cup(B\cap A)\cup(C\cap A)
$$ 
is finite as a finite union of finite sets, which is a contradiction.
\end{proof}

Let us give a couple of examples, which shows how to construct a measure $\mu$ (or a series) with prescribed cardinal function $f_{\mu}$ or with a given range $R(f_\mu)$. The first example in the row shows that there is a measure $\mu$ (or series) which has exactly two points $t\neq v$ with $f_\mu(t)=f_\mu(v)=3$. 

\begin{example}\label{trojkadubleton}
\emph{
Let  $\sum_{n=1}^{\infty}y_n$ be a series such that each point of $A(y_n)$ has a unique representation. Let us define $x_1=x_2:=\sum_{n=1}^{\infty}y_n$ and $x_{n+2}:=y_n$ for each $n\in\mathbb{N}$. 
Clearly $A(x_n)=A(y_n)\cup (A(y_n)+x_1)\cup(A(y_n)+2x_1)$. We have $f^{-1}(\{1\})=(A(y_n)\cup(A(y_n)+2x_1))\setminus\{x_1,2x_1\}$, $f^{-1}(\{2\})=(A(y_n)+x_1)\setminus\{x_1,2x_1\}$ and $f^{-1}(\{3\})=\{x_1,2x_1\}$. 
}
\end{example}
The following illustration describe the cardinal function for Example \ref{trojkadubleton} with scaling $\sum_{n=1}^{\infty}y_n=1$ and $A=A(y_n)$.
\\
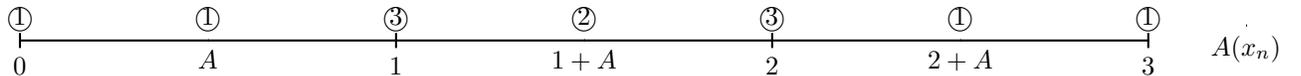
\begin{tikzpicture}[>=stealth,thick] 
 \draw[-] (0,0)--(15,0);
 \draw (2.5,0)--(2.5,0) node[below]{$A$};
 \draw (7.5,0)--(7.5,0) node[below]{$1+A$};
 \draw (12.5,0)--(12.5,0) node[below]{$2+A$};
 \draw (0,0.1)--(0,-0.1) node[below]{$0$};
 \draw (5,0.1)--(5,-0.1) node[below]{$1$};
 \draw (10,0.1)--(10,-0.1) node[below]{$2$};
 \draw (15,0.1)--(15,-0.1) node[below]{$3$};
 \draw (16.3,0.2)--(16.3,0.2) node[below]{$A(x_n)$};
 \draw (2.5,0)--(2.5,0) node[above]{$\textcircled{1}$};
 \draw (7.5,0)--(7.5,0) node[above]{$\textcircled{2}$};
 \draw (12.5,0)--(12.5,0) node[above]{$\textcircled{1}$};
 \draw (0,0)--(0,0) node[above]{$\textcircled{1}$};
 \draw (5,0)--(5,0) node[above]{$\textcircled{3}$};
 \draw (10,0)--(10,0) node[above]{$\textcircled{3}$};
 \draw (15,0)--(15,0) node[above]{$\textcircled{1}$};
\end{tikzpicture}
The picture should be read as follows. Point of $A\cap[0,1)$ and $(2+A)\cap(2,3]$ have unique representations; points of $(1+A)\cap(1,2)$ have two representations, while points 1 and 2 have three representations. 

The following shows that it is possible to find $\mu$ such that $\mu(\N)/2$ has exactly 4 representation, and there are two points with 3 representations. 
\begin{example}\label{czworkasingleton}
\emph{
Let  $\sum_{n=1}^{\infty}y_n$ be a series such that each point of $A(y_n)$ has a unique representation. Let us define $x_1=x_2:=2\sum_{n=1}^{\infty}y_n$, $x_3:=\sum_{n=1}^{\infty}y_n$ and $x_{n+3}:=y_n$ for each $n\in\mathbb{N}$.
The cardinal function is illustrated below.
\\
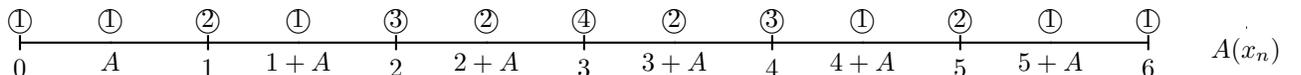
\begin{tikzpicture}[>=stealth,thick] 
 \draw[-] (0,0)--(15,0);
 \draw (1.2,0)--(1.2,0) node[below]{$A$};
 \draw (3.7,0)--(3.7,0) node[below]{$1+A$};
 \draw (6.2,0)--(6.2,0) node[below]{$2+A$};
 \draw (8.7,0)--(8.7,0) node[below]{$3+A$};
 \draw (11.2,0)--(11.2,0) node[below]{$4+A$};
 \draw (13.7,0)--(13.7,0) node[below]{$5+A$};
 \draw (0,0.1)--(0,-0.1) node[below]{$0$};
 \draw (2.5,0.1)--(2.5,-0.1) node[below]{$1$};
 \draw (5,0.1)--(5,-0.1) node[below]{$2$};
 \draw (7.5,0.1)--(7.5,-0.1) node[below]{$3$};
 \draw (10,0.1)--(10,-0.1) node[below]{$4$};
 \draw (12.5,0.1)--(12.5,-0.1) node[below]{$5$};
 \draw (15,0.1)--(15,-0.1) node[below]{$6$};
 \draw (16.3,0.2)--(16.3,0.2) node[below]{$A(x_n)$};
 \draw (0,0)--(0,0) node[above]{$\textcircled{1}$};
 \draw (1.2,0)--(1.2,0) node[above]{$\textcircled{1}$};
 \draw (2.5,0)--(2.5,0) node[above]{$\textcircled{2}$};
 \draw (3.7,0)--(3.7,0) node[above]{$\textcircled{1}$};
 \draw (5,0)--(5,0) node[above]{$\textcircled{3}$};
 \draw (6.2,0)--(6.2,0) node[above]{$\textcircled{2}$};
 \draw (7.5,0)--(7.5,0) node[above]{$\textcircled{4}$};
 \draw (8.7,0)--(8.7,0) node[above]{$\textcircled{2}$};
 \draw (10,0)--(10,0) node[above]{$\textcircled{3}$};
 \draw (11.2,0)--(11.2,0) node[above]{$\textcircled{1}$};
 \draw (12.5,0)--(12.5,0) node[above]{$\textcircled{2}$};
 \draw (13.7,0)--(13.7,0) node[above]{$\textcircled{1}$};
 \draw (15,0)--(15,0) node[above]{$\textcircled{1}$};
\end{tikzpicture}
 }
\end{example}

It is possible to have $R(f_\mu)=\{1,2,4\}$.
\begin{example}\label{czworkasingletonbeztrojki}
\emph{
Let  us consider the following series, a slight modification of Example \ref{czworkasingleton} -- $x_1=x_2:=3\sum_{n=1}^{\infty}y_n$, $x_3:=\sum_{n=1}^{\infty}y_n$ and $x_{n+3}:=y_n$. Its cardinal function is as follows.
\\
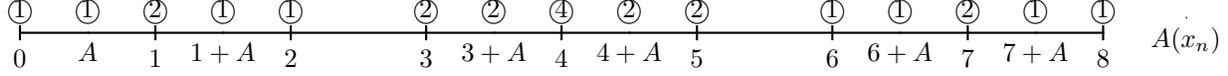
\begin{tikzpicture}[>=stealth,thick] 
 \draw[-] (0,0)--(14.4,0);
 \draw (0.9,0)--(0.9,0) node[below]{$A$};
 \draw (2.7,0)--(2.7,0) node[below]{$1+A$};
 \draw (6.3,0)--(6.3,0) node[below]{$3+A$};
 \draw (8.1,0)--(8.1,0) node[below]{$4+A$};
 \draw (11.7,0)--(11.7,0) node[below]{$6+A$};
 \draw (13.5,0)--(13.5,0) node[below]{$7+A$};
 \draw (0,0.1)--(0,-0.1) node[below]{$0$};
 \draw (1.8,0.1)--(1.8,-0.1) node[below]{$1$};
 \draw (3.6,0.1)--(3.6,-0.1) node[below]{$2$};
 \draw (5.4,0.1)--(5.4,-0.1) node[below]{$3$};
 \draw (7.2,0.1)--(7.2,-0.1) node[below]{$4$};
 \draw (9,0.1)--(9,-0.1) node[below]{$5$};
 \draw (10.8,0.1)--(10.8,-0.1) node[below]{$6$};
 \draw (12.6,0.1)--(12.6,-0.1) node[below]{$7$};
 \draw (14.4,0.1)--(14.4,-0.1) node[below]{$8$};
 \draw (15.5,0.2)--(15.5,0.2) node[below]{$A(x_n)$};
 \draw (0,0)--(0,0) node[above]{$\textcircled{1}$};
 \draw (0.9,0)--(0.9,0) node[above]{$\textcircled{1}$};
 \draw (1.8,0)--(1.8,0) node[above]{$\textcircled{2}$};
 \draw (2.7,0)--(2.7,0) node[above]{$\textcircled{1}$};
 \draw (3.6,0)--(3.6,0) node[above]{$\textcircled{1}$};
 \draw (5.4,0)--(5.4,0) node[above]{$\textcircled{2}$};
 \draw (6.3,0)--(6.3,0) node[above]{$\textcircled{2}$};
 \draw (7.2,0)--(7.2,0) node[above]{$\textcircled{4}$};
 \draw (8.1,0)--(8.1,0) node[above]{$\textcircled{2}$};
 \draw (9,0)--(9,0) node[above]{$\textcircled{2}$};
 \draw (10.8,0)--(10.8,0) node[above]{$\textcircled{1}$};
 \draw (11.7,0)--(11.7,0) node[above]{$\textcircled{1}$};
 \draw (12.6,0)--(12.6,0) node[above]{$\textcircled{2}$};
 \draw (13.5,0)--(13.5,0) node[above]{$\textcircled{1}$};
 \draw (14.4,0)--(14.4,0) node[above]{$\textcircled{1}$};
\end{tikzpicture}
}
\end{example}

It is possible also to have $R(f_{\mu})=\{1,3\}$. 
\begin{example}\label{trojkabezdwojki}
\emph{
Put  $x_1=x_2=x_3:=2\sum_{n=1}^{\infty}y_n$ and $x_{n+3}:=y_n$ for each $n\in\mathbb{N}$.
Its cardinal function is described below.
\\
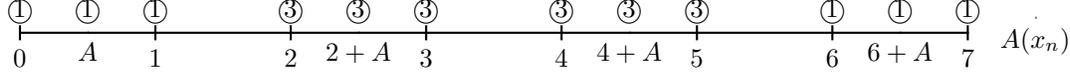
\begin{tikzpicture}[>=stealth,thick] 
 \draw[-] (0,0)--(12.6,0);
 \draw (0.9,0)--(0.9,0) node[below]{$A$};
 \draw (4.5,0)--(4.5,0) node[below]{$2+A$};
 \draw (8.1,0)--(8.1,0) node[below]{$4+A$};
 \draw (11.7,0)--(11.7,0) node[below]{$6+A$};
 \draw (0,0.1)--(0,-0.1) node[below]{$0$};
 \draw (1.8,0.1)--(1.8,-0.1) node[below]{$1$};
 \draw (3.6,0.1)--(3.6,-0.1) node[below]{$2$};
 \draw (5.4,0.1)--(5.4,-0.1) node[below]{$3$};
 \draw (7.2,0.1)--(7.2,-0.1) node[below]{$4$};
 \draw (9,0.1)--(9,-0.1) node[below]{$5$};
 \draw (10.8,0.1)--(10.8,-0.1) node[below]{$6$};
 \draw (12.6,0.1)--(12.6,-0.1) node[below]{$7$};
 \draw (13.5,0.2)--(13.5,0.2) node[below]{$A(x_n)$};
 \draw (0,0)--(0,0) node[above]{$\textcircled{1}$};
 \draw (0.9,0)--(0.9,0) node[above]{$\textcircled{1}$};
 \draw (1.8,0)--(1.8,0) node[above]{$\textcircled{1}$};
 \draw (3.6,0)--(3.6,0) node[above]{$\textcircled{3}$};
 \draw (4.5,0)--(4.5,0) node[above]{$\textcircled{3}$};
 \draw (5.4,0)--(5.4,0) node[above]{$\textcircled{3}$};
 \draw (7.2,0)--(7.2,0) node[above]{$\textcircled{3}$};
 \draw (8.1,0)--(8.1,0) node[above]{$\textcircled{3}$};
 \draw (9,0)--(9,0) node[above]{$\textcircled{3}$};
 \draw (10.8,0)--(10.8,0) node[above]{$\textcircled{1}$};
 \draw (11.7,0)--(11.7,0) node[above]{$\textcircled{1}$};
 \draw (12.6,0)--(12.6,0) node[above]{$\textcircled{1}$};
\end{tikzpicture}
}
\end{example}

An it is possible to have $R(f_{\mu})=\{1,3,4\}$. 
\begin{example}\label{czworkabezdwojki}
\emph{
Define  $x_1:=4\sum_{n=1}^{\infty}y_n$, $x_2=x_3=x_4:=2\sum_{n=1}^{\infty}y_n$ and $x_{n+4}:=y_n$ for each $n\in\mathbb{N}$. We have:
\\
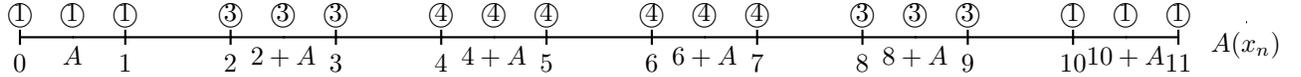
\begin{tikzpicture}[>=stealth,thick] 
 \draw[-] (0,0)--(15.4,0);
 \draw (0.7,0)--(0.7,0) node[below]{$A$};
 \draw (3.5,0)--(3.5,0) node[below]{$2+A$};
 \draw (6.3,0)--(6.3,0) node[below]{$4+A$};
 \draw (9.1,0)--(9.1,0) node[below]{$6+A$};
 \draw (11.9,0)--(11.9,0) node[below]{$8+A$};
 \draw (14.7,0)--(14.7,0) node[below]{$10+A$};
 \draw (0,0.1)--(0,-0.1) node[below]{$0$};
 \draw (1.4,0.1)--(1.4,-0.1) node[below]{$1$};
 \draw (2.8,0.1)--(2.8,-0.1) node[below]{$2$};
 \draw (4.2,0.1)--(4.2,-0.1) node[below]{$3$};
 \draw (5.6,0.1)--(5.6,-0.1) node[below]{$4$};
 \draw (7,0.1)--(7,-0.1) node[below]{$5$};
 \draw (8.4,0.1)--(8.4,-0.1) node[below]{$6$};
 \draw (9.8,0.1)--(9.8,-0.1) node[below]{$7$};
 \draw (11.2,0.1)--(11.2,-0.1) node[below]{$8$};
 \draw (12.6,0.1)--(12.6,-0.1) node[below]{$9$};
 \draw (14,0.1)--(14,-0.1) node[below]{$10$};
 \draw (15.4,0.1)--(15.4,-0.1) node[below]{$11$};
 \draw (16.3,0.2)--(16.3,0.2) node[below]{$A(x_n)$};
 \draw (0,0)--(0,0) node[above]{$\textcircled{1}$};
 \draw (0.7,0)--(0.7,0) node[above]{$\textcircled{1}$};
 \draw (1.4,0)--(1.4,0) node[above]{$\textcircled{1}$};
 \draw (2.8,0)--(2.8,0) node[above]{$\textcircled{3}$};
 \draw (3.5,0)--(3.5,0) node[above]{$\textcircled{3}$};
 \draw (4.2,0)--(4.2,0) node[above]{$\textcircled{3}$};
 \draw (5.6,0)--(5.6,0) node[above]{$\textcircled{4}$};
 \draw (6.3,0)--(6.3,0) node[above]{$\textcircled{4}$};
 \draw (7,0)--(7,0) node[above]{$\textcircled{4}$};
 \draw (8.4,0)--(8.4,0) node[above]{$\textcircled{4}$};
 \draw (9.1,0)--(9.1,0) node[above]{$\textcircled{4}$};
 \draw (9.8,0)--(9.8,0) node[above]{$\textcircled{4}$};
 \draw (11.2,0)--(11.2,0) node[above]{$\textcircled{3}$};
 \draw (11.9,0)--(11.9,0) node[above]{$\textcircled{3}$};
 \draw (12.6,0)--(12.6,0) node[above]{$\textcircled{3}$};
 \draw (14,0)--(14,0) node[above]{$\textcircled{1}$};
 \draw (14.7,0)--(14.7,0) node[above]{$\textcircled{1}$};
 \draw (15.4,0)--(15.4,0) node[above]{$\textcircled{1}$};
\end{tikzpicture}
}
\end{example}

The next two propositions give methods to produce a measure $\mu$ with prescribed cardinal properties having other measure $\nu$ with given cardinal properties.

Let $\nu$ be a measure. Put $E:=R(f_\nu)$ and $F:=f_\nu[\on{rng}(\nu)\setminus\{0,\nu(\N)\}]$. Note that $F=E\setminus\{1\}$ if and only if the series $\sum_{n=1}^{\infty} x_n$ is plentiful where $x_n:=\nu(\{n\})$. As usual $tA:=\{ta:a\in A\}$ for any $t\in\R$ and $\subset\R$. 
\begin{proposition}\label{rozmnazaniefunkcjikardynalnych}
Let $\nu$ be a measure and $m\in\N$. There exist measures $\mu_i$ with cardinal functions $g_i$, $i=1,2,3,4$, such that
\begin{enumerate}
\item $R(g_1)=E\cup\{2\}$;
\item $R(g_2)=E\cup\{3\}\cup 2F$;
\item $R(g_3)= \bigcup_{k=0}^{m} {m\choose k}F\cup \{{m+1\choose k}:k=0,1,\dots,m+1\}$;
\item $R(g_4)= \bigcup_{k=0}^{m} {m\choose k}E$.
\end{enumerate}
\end{proposition}
\begin{proof} Let $x_n$ be the value of $n$-th $\nu$-atom. \\
(1) The construction is the same as in Example \ref{dwojkawsrodku} -- now the measure $\nu$ is arbitrary. Define $y_1:=\sum_{n=1}^{\infty} x_n$ and $y_{n+1}:=x_n$ for each natural number $n$. Then $\on{rng}(\mu_1)=A(y_n)=A(x_n)\cup(\sum_{n=1}^{\infty} x_n+A(x_n))$ and we have:
\begin{displaymath}
g_1(x) = \left\{ \begin{array}{ll} f(x), & \textrm{$x\in A(x_n)\setminus\{\sum_{n=1}^{\infty} x_n\}$}\\ f(x-\sum_{n=1}^{\infty} x_n), & \textrm{$x\in (\sum_{n=1}^{\infty} x_n+A(x_n))\setminus\{\sum_{n=1}^{\infty} x_n\}$}\\ 2, & \textrm{$x=\sum_{n=1}^{\infty} x_n$} \end{array} \right.
\end{displaymath}
(2) Here we have the same construction as in Example \ref{trojkadubleton}. Define $y_1=y_2:=\sum_{n=1}^{\infty} x_n$ and $y_{n+2}:=x_n$ for all $n\in\mathbb{N}$. Thus $\on{rng}(\mu_2)=A(y_n)=A(x_n)\cup(\sum_{n=1}^{\infty} x_n+A(x_n))\cup (2\sum_{n=1}^{\infty} x_n+A(x_n))$. The cardinal function is as follows:
\begin{displaymath}
g_2(x) = \left\{ \begin{array}{ll} f(x), & \textrm{$x\in A(x_n)\setminus\{\sum_{n=1}^{\infty} x_n\}$}\\ 2f(x-\sum_{n=1}^{\infty} x_n), & \textrm{$x\in (\sum_{n=1}^{\infty} x_n+A(x_n))\setminus\{\sum_{n=1}^{\infty} x_n,2\sum_{n=1}^{\infty} x_n\}$}\\ 2f(x-2\sum_{n=1}^{\infty} x_n), & \textrm{$x\in (2\sum_{n=1}^{\infty} x_n+A(x_n))\setminus\{2\sum_{n=1}^{\infty} x_n\}$} \\ 3, & \textrm{$x\in\{\sum_{n=1}^{\infty} x_n,2\sum_{n=1}^{\infty} x_n\}$}\end{array} \right.
\end{displaymath}
(3)   Define $y_n:=\sum_{n=1}^{\infty} x_n$ for all $n\leq m$ and $y_{n}:=x_{n-m}$ for all $n>m$. Then $\on{rng}(\mu_3)=A(y_n)= \bigcup_{k=0}^{m} k\sum_{n=1}^{\infty} x_n+A(x_n)$. We have $g_3(k\sum_{n=1}^{\infty} x_n)={m+1\choose k}$ for every  $k\in\{0,1,\ldots,m+1\}$ and $g_3(k\sum_{n=1}^{\infty} x_n+x)={m\choose k} f(x)$ for all $x\in A(x_n)\setminus\{0,\sum_{n=1}^{\infty} x_n\}$ and $k\in\{0,1,\ldots,m\}$. Hence $R(g_3)= \bigcup_{k=0}^{m} {m\choose k}F\cup\{{m+1\choose k}:k=0,1,\dots,m+1\}$.
\\(4) Fix $m\in\mathbb{N}$. Define $y_n:=2\sum_{n=1}^{\infty} x_n$ for all $n\leq m$ and $y_{n}:=x_{n-m}$ for all $n>m$.
Then $\on{rng}(\mu_4)=A(y_n)= \bigcup_{k=0}^{m} \left(2k\sum_{n=1}^{\infty} x_n+A(x_n)\right)$. We have $g_3(2k\sum_{n=1}^{\infty} x_n+x)={m\choose k} f(x)$ for all $x\in A(x_n)$ and $k\in\{0,1,\ldots,m\}$. Hence $R(g_3)= \bigcup_{k=0}^{m} {m\choose k}E$.
\end{proof}

Let us recall one basic fact about arithmetic of cardinal numbers. If $\kappa$ is infinite cardinal number and $k\in\N$, then $k\cdot \kappa=\kappa\cdot k=\kappa$. If $\kappa$ is a natural number, then $\kappa\cdot k$ is a usual product of two natural numbers. Therefore $k\cdot\omega=\omega$ and $k\cdot\mathfrak{c}=\mathfrak{c}$. 
\begin{proposition} \label{szeregiskonczone}
Let $\tau$ be a measure on $[1,j]$, and let $\nu$ be a measure on $\N$. Then there is a measure $\mu$ such that $R(f_\mu)=R(f_{\tau})\cdot R(f_\nu)=\{k\cdot\kappa:k\in\on{rng}(\tau),\kappa\in\on{rng}(\nu)\}$ holds.
\end{proposition}
\begin{proof} Let $k_i:=\tau(\{i\})$ for $i\in[1,j]$ and $y_n:=\nu(\{n\})$ for $n\in\N$. Let $\on{rng}(\tau):=\{\sigma_1,\ldots,\sigma_m\}$. If $(a\nu)(E):=a\cdot\nu(E)$, then $R(f_{a\nu})=R(f_\nu)$. Hence we may assume that $\sum_{n=1}^{\infty}y_n<\min\{\sigma_i-\sigma_l : \ i,l\in\{1,\ldots,m\}, \  i\neq l\}$. Let us define $x_i:=k_i$ for $i\in[1,j]$ and $x_{j+n}:=y_n$ for all $n\in\mathbb{N}$. Thus $\on{rng}(\mu)=\on{rng}(\tau)+\on{rng}(\nu)=\bigcup_{p=1}^{m} (\sigma_p+\on{rng}(\nu))$ and $(\sigma_i+\on{rng}(\nu))\cap (\sigma_l+\on{rng}(\nu))=\emptyset$ for all $i,l\in\{1,\ldots,m\}$, $i\neq l$. As a result we have $f(\sigma_p+y)=f_{\sigma}(\sigma_p)\cdot f_y(y)$  for all $p\in\{1,\ldots,m\}$, $y\in \on{rng}(\nu)$. Hence $R(f_\mu)=R(f_{\tau})\cdot R(f_\nu)$.
\end{proof}

Note that to produce a measure $\mu$ on $\N$ with given range $R=R(f_\mu)\subset\N$, it is enough to find a measure $\tau$ on some finite set $[1,n]$ with $R(f_\tau)=R$. To see it take a measure $\nu$ on $\N$ such that any value of $\on{rng}(\nu)$ has a unique representation. Then by Proposition \ref{rozmnazaniefunkcjikardynalnych} we find $\mu$ with $R(f_\mu)=R(f_\tau)\cdot 1=R$. On the other hand if $\nu$ is such that $R(f_\nu)=\{1,\mathfrak{c}\}$, see Example \ref{skrajne}, then there is $\mu$ with $R(f_\mu)=R(f_\tau)\cdot \{1,\mathfrak{c}\}=R\cup\{\mathfrak{c}\}$.  

The following Table summarize our searching of measures $\mu$ with $R(f_\mu)\subset[1,6]$. Some of them can be found by Proposition \ref{rozmnazaniefunkcjikardynalnych}, while for the others we need to produce new examples. 

\begin{center}
\begin{tabular}{c|c}
Range $R(f_\mu)$ & method\\ \hline
1,2 & Prop. \ref{rozmnazaniefunkcjikardynalnych} (1)  \\ 
1,2,3 & Prop. \ref{rozmnazaniefunkcjikardynalnych} (2)  \\ 
1,3 & Prop. \ref{rozmnazaniefunkcjikardynalnych} (4) for $m=3$  \\ 
1,2,3,4 & Prop. \ref{rozmnazaniefunkcjikardynalnych} (2) for $\nu$ with $R(f_\nu)=\{1,2\}$  \\ 
1,2,4 & Prop. \ref{rozmnazaniefunkcjikardynalnych} (4) for  $m=2$ and $\nu$ with $R(f_\nu)=\{1,2\}$  \\ 
1,3,4 & Example \ref{czworkabezdwojki} \\
1,2,3,5 & Example \ref{jedendwatrzypiec}\\
1,2,4,5 & Example \ref{jedendwaczterypiec}\\
1,4,6 & Prop. \ref{rozmnazaniefunkcjikardynalnych} (4) for  $m=4$\\ 
1,3,4,6 & Prop. \ref{rozmnazaniefunkcjikardynalnych} (3) for  $m=3$\\
1,4,5,6 & Example \ref{jedenczterypiecszesc}\\
1,3,4,5,6 & Example \ref{jedentrzyczterypiecszesc}\\ 
1,3,6 & Example \ref{jedentrzyszesc}\\
\hline
\end{tabular}
\end{center}

Examples \ref{jedentrzyszesc}-\ref{jedentrzypiecsiedem} are constructed in a similar way that in Proposition \ref{szeregiskonczone}. We assume that  $\sum_{n=1}^{\infty}y_n$ is a series such that any its point has unique representation, and $f$ is the cardinal function for the final series $\sum_{n=1}^{\infty}x_n$.

\begin{example}\label{jedentrzyszesc}
Let $x_1=x_2:=6\sum_{n=1}^{\infty}y_n$, $x_3=x_4=x_5:=2\sum_{n=1}^{\infty}y_n$ and $x_{n+5}:=y_n$ for each $n\in\mathbb{N}$. Then $R(f)=\{1,3,6\}$. 
\end{example}

\begin{example}\label{jedenczterypiecszesc}
Let $x_1:=12\sum_{n=1}^{\infty}y_n$, $x_2:=6\sum_{n=1}^{\infty}y_n$, $x_3=x_4=x_5=x_6:=2\sum_{n=1}^{\infty}y_n$ and $x_{n+6}:=y_n$ for each $n\in\mathbb{N}$. Then $R(f)=\{1,4,5,6\}$. 
\end{example}

\begin{example}\label{jedentrzyczterypiecszesc}
Let  $x_1:=6\sum_{n=1}^{\infty}y_n$, $x_2:=4\sum_{n=1}^{\infty}y_n$,$x_3=x_4=x_5:=2\sum_{n=1}^{\infty}y_n$ and $x_{n+5}:=y_n$ for each $n\in\mathbb{N}$. Then $R(f)=\{1,3,4,5,6\}$. 
\end{example}

\begin{example}\label{jedendwatrzypiec}
Let $x_1:=6\sum_{n=1}^{\infty}y_n$, $x_2=x_3:=4\sum_{n=1}^{\infty}y_n$, $x_4=x_5:=2\sum_{n=1}^{\infty}y_n$ and $x_{n+5}:=y_n$ for each $n\in\mathbb{N}$.  Then $R(f)=\{1,2,3,5\}$. 
\end{example}

\begin{example}\label{jedendwaczterypiec}
Let $x_1:=10\sum_{n=1}^{\infty}y_n$, $x_2=x_3:=6\sum_{n=1}^{\infty}y_n$, $x_4=x_5:=4\sum_{n=1}^{\infty}y_n$ and $x_{n+5}:=y_n$ for each $n\in\mathbb{N}$.  Then $R(f)=\{1,2,4,5\}$. 
\end{example}

\begin{example}\label{jedentrzypiecsiedem}
Let $x_1=x_2:=4\sum_{n=1}^{\infty}y_n$, $x_3=x_4=x_5:=2\sum_{n=1}^{\infty}y_n$ and $x_{n+5}:=y_n$ for each $n\in\mathbb{N}$.  Then $R(f)=\{1,3,5,7\}$. 
\end{example}

The above table does not contain some subsets of $[1,6]$ (clearly each range has to contain 1).  For example we were not able to produce a measure $\mu$ with $R(f_\mu)=\{1,4\}$. We have found that there is no measure $\nu$ defined on a finite set with $R(f_\nu)=\{1,4\}$, but we do not know whether there is a measure $\mu$ on $\N$ with $R(f_\mu)=\{1,4\}$. On the other hand we found that if such measure exists, then the mapping $n\mapsto\mu(\{n\})$ has to be one-to-one. 

\begin{proposition}\label{jedenczterynoninjective}
Suppose that $\mu$ is a measure with $R(f_\mu)=\{1,4\}$. Then $\mu(\{n\})\neq\mu(\{m\})$ for distinct $n,m\in\N$
\end{proposition}
\begin{proof}Let $x_n:=\mu(\{n\})$. 
Assume that $x_k=x_m$. Since $f(x_k)\geq 2$, we have $f(x_k)=4$, that is $x_k=x_m=\mu(A)=\mu(B)$, where $A\neq B$, $\{k\}\neq A\neq \{m\}$ and $\{k\}\neq B\neq \{m\}$. Hence $x_k+x_m=\mu( A\cup\{k\})=\mu( A\cup\{m\})=\mu( B\cup\{k\})=\mu( B\cup\{m\})$, so $f(x_k+x_m)\geq 5$, which gives contradiction.
\end{proof}

\end{document}